\documentclass[a4paper]{article}
\usepackage{graphicx, amssymb,xcolor,amsmath,amsthm,soul,hyperref,dsfont}
\hypersetup{hidelinks,
 pdfstartview=FitH,
 colorlinks=true,
 linkcolor=blue,
 citecolor=blue,
 urlcolor=blue}
\usepackage[T2A]{fontenc}
\usepackage{float}
\usepackage{cleveref}
\crefname{equation}{}{}
\usepackage{thmtools, thm-restate}
\usepackage{bbm}
\usepackage{mleftright}
\mleftright
\usepackage{microtype}

\newcommand\blfootnote[1]{
  \begingroup
  \renewcommand\thefootnote{}\footnote{#1}
  \addtocounter{footnote}{-1}
  \endgroup
}

\newcommand{\ind}{\mathbbm{1}}

\DeclareMathOperator{\Var}{Var}

\newcommand{\eps}{\varepsilon}
\newtheorem{theorem}{Theorem}
\newtheorem{remark}[theorem]{Remark}
\newtheorem{lemma}[theorem]{Lemma}
\newtheorem{problem}{Problem}
\newtheorem{corollary}[theorem]{Corollary}

\newtheorem{claim}[theorem]{Claim}

\newtheorem*{theorem*}{Theorem}
\theoremstyle{remark}
\newtheorem*{remark*}{Remark}

\title{Recurrence, transience and anti-concentration of Rademacher random walks}
\author{Satyaki Bhattacharya\thanks{Centre for Mathematical Sciences, Lund University, Box 118 SE-22100, Lund, Sweden.} \and Edward Crane\thanks{School of Mathematics, University of Bristol, Bristol, BS8\thinspace1UG, UK and Heilbronn Institute for Mathematical Research, Bristol, UK.} \and Tom Johnston\footnotemark[2]}

\begin{document}

\maketitle

\begin{abstract}
    The Rademacher random walk associated with a deterministic sequence $(a_n)_{n \geq 1}$ is the walk which starts at zero and, at step $i$, independently steps either up or down by $a_i$ with equal probability.
    We continue the study begun by Bhattacharya and Volkov in 2023 of the transience or recurrence of one-dimensional Rademacher random walks. In particular, we show that if the sequence of step sizes is bounded, the walk is weakly recurrent, meaning that it returns infinitely often to a random finite interval, while if the step sizes tend to infinity arbitrarily slowly, the walk may be transient. On the other hand, using a construction with integer step sizes, we show that the step sizes may grow arbitrarily fast and still give a weakly recurrent random walk. We also show, using a construction with non-integer step sizes, that the same conclusion holds even if we restrict to strictly increasing step sizes.
    However, we prove that if $a_n = n^{\alpha + o(1)}$ for some $\alpha > 1/2$, then the walk is transient. 
    We show that the bound on $\alpha$ is tight by giving an example where $a_n = \Theta(n^{1/2})$ and the walk is weakly recurrent.
   \blfootnote{Email addresses:
    \href{mailto:satyaki.bhattacharya@matstat.lu.se}{\nolinkurl{satyaki.bhattacharya@matstat.lu.se}},
    \href{mailto:edward.crane@bristol.ac.uk}{\nolinkurl{edward.crane@bristol.ac.uk}},
    \href{mailto:tom.johnston@bristol.ac.uk}{\nolinkurl{tom.johnston@bristol.ac.uk}}.
    }
\end{abstract}


\section{Introduction}

Throughout this paper, $\epsilon_1, \epsilon_2, \dots$ is a sequence of independent Rademacher($1/2$) random variables, taking the values $\pm1$ each with probability $1/2$. We will refer to these simply as Rademacher random variables (dropping the $1/2$). Let $(a_n)_{n \geq 1}$ be a deterministic sequence of non-negative real numbers, and define the associated \emph{Rademacher random walk} $(X_n)_{n \ge 0}$ by 
\[ X_n = \sum_{k=1}^n \epsilon_k a_k.\] 
When all of the $a_k$ take the value 1, the associated Rademacher random walk is exactly the simple symmetric random walk in one dimension, and the recurrence of this random walk is a fundamental result in any introductory probability course.
However, by simply changing the sequence of step sizes $(a_n)$, the problem becomes much less elementary and the behaviour of the Rademacher random walk is still unknown in many cases, including the simple-looking case where $a_n = n^{\alpha}$. 

 In this paper we will focus on the one-dimensional case described above, continuing the study begun by Bhattacharya and Volkov in \cite{bhattacharya2023recurrence}; see also the recent monograph by Engl\"{a}nder and Volkov~\cite[Chapter~8]{englander2025coin}. 
We mention that recurrence and transience of two-dimensional generalization of Rademacher random walks has been studied recently by Bhattacharya and Volkov in~\cite{bhattacharya2025twodimensionalrademacherwalk}.

If a particular sample of a Rademacher random walk satisfies $\{X_n\leq C\}$ infinitely often (i.o.), we call it \emph{C-recurrent}. When $X$ is almost surely $0$-recurrent we say $X$ is \emph{recurrent}. Note that $C$-recurrence is not necessarily a tail event. Indeed, if we take the sequence $(a_n)_{n \geq 1}$ where $a_1 = 1$, $a_2 = 1$ and $a_i = 3$ for all $i \geq 2$, then whether the walk returns to zero i.o.~depends on the value after two steps.
However, the event $\{\exists C : |X_n| \leq C~i.o.\}$ is a tail event, and we shall call a Rademacher random walk \emph{weakly recurrent} if the probability of this event is 1. Otherwise, we will call the walk \emph{transient}. Finally, we say that a random walk is \emph{topologically recurrent} if its range is almost surely dense in $\mathbb{R}$.

Note that weak recurrence and topological recurrence of the Rademacher random walk are both unaffected by making any finite number of changes, deletions, or insertions to  $(a_n)_{n \ge 1}$; in other words they depend only on the tail of the sequence.   

One interesting class of Rademacher random walks is those whose step sizes $(a_n)_{n \geq 1}$ satisfy $\sum_{n} a_n^2 < \infty$ but $\sum_n a_n = \infty$.
In this case, $X_n$ almost surely converges to a random limit, by Kolmogorov's two series theorem. 
However, the distribution of the limit has unbounded support, so $X$ is weakly recurrent but there is no $C < \infty$ such that $X$ is almost surely $C$-recurrent. 

Our main motivation is the following problem.
\begin{problem}\label{prob:main}
    What conditions on the growth rate of $(a_n)$ guarantee that the associated Rademacher random walk is transient? What growth conditions guarantee that the Rademacher random walk is weakly recurrent?
\end{problem}

In \cite{bhattacharya2023recurrence} it is shown that the Rademacher random walk is $C$-recurrent for every $C > 0$,  and hence weakly recurrent, in the following cases:
\begin{itemize}
\item  $a_n = \log n$, and
\item  $a_n = \lfloor (\log_\gamma n)^\beta\rfloor$ for constants $\gamma > 1$ and $0 < \beta \le 1$.
\end{itemize}
Two transient cases with integer step sizes are also given:
\begin{itemize}
\item the steps $(a_n)_{n \ge 1}$ are all distinct integers, and
\item $a_n = \lfloor n^\beta \rfloor$ for any constant $0 < \beta < 1$.
\end{itemize}

One particularly natural case of Problem~\ref{prob:main} is to assume that $a_n \approx n^{\alpha}$ for some $\alpha > 0$. With this constraint, we can ask whether there are values of $\alpha$ for which we can guarantee that the walk is transient or weakly recurrent.
Combining the two transience results from \cite{bhattacharya2023recurrence}, we see that the Rademacher walk with step sizes $a_n = \lfloor n^{\alpha}\rfloor$ is transient for all $\alpha > 0$. Hence, there is no $\alpha > 0$ for which the condition $a_n \approx n^\alpha$ implies the walk is weakly recurrent.
In contrast, our main result shows that there are values of $\alpha$ for which this growth rate guarantees that the walk is transient.

\begin{theorem}
\label{thm:transient}
    Let $(a_n)_{n \geq 1}$ be a sequence and suppose that $a_n  = n^{\alpha + o(1)}$ for some $\alpha > 1/2$. Then the associated Rademacher random walk is transient.
\end{theorem}

The condition on $\alpha$ in Theorem~\ref{thm:transient} is best possible, even under the stronger assumption that $a_n = \Theta( n^{\alpha})$ for some $\alpha > 0$ as $n \to \infty$: there is a sequence of step sizes $(a_n)_{n \geq 1}$ where $a_n = \Theta(n^{1/2})$ and the associated Rademacher random walk is weakly recurrent.

\begin{theorem}
\label{thm:recurrent}
    There exists a sequence $(a_n)_{n \geq 1}$ of integers such that $a_n = \Theta(n^{1/2})$ as $n \to \infty$ and the associated Rademacher random walk is weakly recurrent.
\end{theorem}

We remark that from our proof of Theorem~\ref{thm:transient} we can already relax the condition that $a_n = n^{\alpha + o(1)}$ slightly to allow $n^{\alpha -\delta} \leq a_n \leq n^{\alpha + \delta}$ for some sufficiently small $\delta = \delta(\alpha)$. However, we still need both upper and lower bounds on $a_n$.
Perhaps surprisingly, the upper bound is necessary: the sequence $(a_n)$ can grow arbitrarily fast and still yield a  recurrent Rademacher walk.
\begin{restatable}{theorem}{fast}\label{thm:fast}
Let $f: \mathbb{N} \to \mathbb{R}$ be any non-decreasing function. There is an integer sequence $(a_n)_{n \geq 1}$ such that $a_n \geq f(n)$ for all $n$ and the associated Rademacher random walk is recurrent. 
\end{restatable}
Our proof of Theorem~\ref{thm:fast} constructs the sequence in blocks of increasing length such that within each block the terms alternate between two consecutive integers. Allowing non-integer step sizes, we can get topological recurrence from a strictly increasing sequence that grows as fast as we like:
\begin{restatable}{theorem}{fastincreasing}\label{thm:fastincreasing}
Let $f: \mathbb{N} \to \mathbb{R}$ be any non-decreasing function. There is a strictly increasing real sequence $(a_n)_{n \geq 1}$ such that $a_n \geq f(n)$ for all $n$ and the associated Rademacher random walk is topologically recurrent. 
\end{restatable}
Recall that Bhattacharya and Volkov showed that if $(a_n)$ consists of distinct integers then the Rademacher random walk is transient. In particular,  the walk associated to any strictly increasing integer sequence $(a_n)$ is transient, and there cannot be a single sequence that demonstrates both Theorem~\ref{thm:fast} and Theorem~\ref{thm:fastincreasing}.

These results leave open the possibility that there could be a powerful sufficient condition for transience of the Rademacher random walk, in terms of the asymptotic behaviour of the sequence $(a_n)$, when we restrict to non-decreasing integer sequences. The asymptotically fastest-growing non-decreasing integer sequences that we know to yield weakly recurrent Rademacher random walks are the examples $a_n = \lfloor c \log n \rfloor$ from \cite{bhattacharya2023recurrence}.  On the other hand, increasing the growth rate slightly, but still using all the non-negative integers, we can obtain a transient Rademacher random walk.

\begin{restatable}{theorem}{incInts}\label{thm:ints}
Let $(a_i)_{i \geq 1}$ be a non-decreasing sequence of integers, and let $L_n$  be the number of times $n$ appears in the sequence.
    Suppose that for some $\eps > 0$ and for all large enough $n$,  
    \begin{equation}
    \label{eqn:assump1}
    \sum_{\overset{m < n}{ \gcd(m,n) = 1}} L_m \geq 2 n^2 \end{equation}
   and \begin{equation}\label{eqn:assump2}
    \sum_{m=1}^{n-1} m^2 L_m \geq  n^2 \log^{3+\eps}(n) \max(4L_n, n^2). \end{equation}
    Then the Rademacher random walk $X$ associated with $(a_i)_{i \geq 1}$ is transient.
\end{restatable}

\begin{restatable}{corollary}{logsq} 
\label{cor:logsquared}
    For any $\alpha > 1$, the Rademacher random walk with step sizes $a_n = \lfloor \log^\alpha(n)\rfloor$ is transient.
\end{restatable}

The (L\'evy) \emph{concentration function} $Q$ of a real valued random variable $A$ is defined by\footnote{We warn the reader that authors disagree about the strictness of the inequalities in the definition of the concentration function, so care is needed in interpreting anti-concentration inequalities in the literature. We have made the same choice as in \cite{juskevicius2024sharp}, because it gives $Q_1(A)$ the meaning that we want in the case of an integer-valued random variable $A$.} 
\[ Q_r(A) = \sup_{x \in \mathbb{R}} \mathbb{P}(x < A \le x+r). \]
 Any upper bound for a value of the concentration function is called an \emph{anti-concentration} bound, while a lower bound is called a \emph{concentration} bound. Some known anti-concentration bounds for sums of independent random variables are discussed in Section~\ref{SS:anti-concentration}.
 
To prove Theorem~\ref{thm:transient} we show, using the theorem below, that the position of the Rademacher random walk at each time $n$ is sufficiently anti-concentrated, and then apply the Borel--Cantelli lemma. We remark that, as alluded to earlier, the following theorem allows us to slightly relax the condition that $a_n = n^{\alpha + o(1)}$.

\begin{restatable}{theorem}{anti}
    \label{thm:anti-concentration}
    Let $(a_n)_{n \geq 0}$ be a sequence and suppose that there are constants $c, C, \alpha > 0$ and $\delta \ge 0$ such that, for all large enough $n$,
    \[cn^{\alpha} \leq a_n \leq Cn^{\alpha + \delta}.\]
     Then, for any $\gamma > 0$, we have the anti-concentration bound
    \[Q_1\left(\sum_{i=1}^n \epsilon_i a_i\right) = O\left(n^{-(\frac{1}{2} + \alpha f(\alpha, \delta) - \gamma)}\right),\]
    where
    \[f(\alpha, \delta) = \begin{cases}
        \frac{\alpha^2}{(\alpha + \delta) (\alpha + 2 \delta + 2 \sqrt{\delta^2 + \alpha \delta})} & \text{if } \delta \leq \frac{\sqrt{\alpha^2 + 1} - \alpha}{2},\\
        \frac{\alpha^2}{(\alpha + \delta)(1 + 2 \delta) (\alpha + 1/2 + \delta)} & \text{if } \delta \geq \frac{\sqrt{\alpha^2 + 1} - \alpha}{2}.
    \end{cases}\]
\end{restatable}

From this it is easy to deduce Theorem~\ref{thm:transient}.
\begin{proof}[Proof of Theorem~\ref{thm:transient} from Theorem~\ref{thm:anti-concentration}]
    Fix $\lambda > 0$ which is small enough that $1/2 + \alpha - \lambda > 1$. For any $\delta > 0$, we have \[n^{\alpha - \delta} \leq a_n \leq n^{\alpha + \delta}\] for all sufficiently large $n$. 
     As $f(\alpha, \delta) \to 1$ as $\delta \to 0$, there is some $\delta > 0$ such that this condition on $a_n$ is enough to get the anti-concentration bound $Q_1(X_n) = O(n^{- (1/2 + \alpha - \lambda)})$.
    
    For any fixed $C$, the probability that $|X_n| \leq C$ is $O(n^{-(1/2 + \alpha - \lambda)})$ and this is summable by our choice of $\lambda$.
    Hence, by the Borel--Cantelli lemma, the probability that $|X_n| \leq C$ infinitely often is 0, and the walk is transient.
\end{proof}

We remark that taking $\delta = 0$ in Theorem~\ref{thm:anti-concentration}, we get the bound $O(n^{-(\alpha + 1/2 - \gamma)})$, which is easily seen to be tight up to the $\gamma$ term by considering the sequence $a_n = n^{\alpha}$. Also of interest is what happens as $\delta \to \infty$. In this case, the anti-concentration gets close to $O(n^{-1/2})$, and it is also straightforward to show that there must be points where this is the correct behaviour. More generally, we have the following easy lower bounds which complement Theorem~\ref{thm:anti-concentration}.
\begin{restatable}{proposition}{lowerAnti} 
\label{prop:lower-anti}
    Fix $\alpha > 0$. Then
    \[Q_1\left(\sum_{i=1}^n \epsilon_i n^{\alpha} \right) = \Omega\left( n^{-(1/2 + \alpha)}\right).\]
    Moreover, for any $\delta \geq 1/2$, there exists a sequence $(a_n)_{n \geq 1}$ of step sizes and a sequence of times $(n_i)_{i \geq 1}$ such that 
    \[n^{\alpha} \le a_n \le n^{\alpha+\delta}\] 
    for all sufficiently large $n$, and as $n \to \infty$,
    \[Q_1\left(\sum_{i=1}^{n_i}\epsilon_i a_i \right) \geq n_i^{-(1/2 + \frac{\alpha}{2\delta} + o(1))}.\]
\end{restatable}

Let us now turn to the second part of Problem~\ref{prob:main}: are there conditions on the growth rate of $(a_n)$ that guarantee that the associated Rademacher random walk is weakly recurrent?

There is no unbounded non-decreasing function $f$ for which the condition $a_n \le f(n)$ suffices to imply weak recurrence, as the following example shows.
Let $f: \mathbb{N} \to [1,\infty)$ be unbounded and non-decreasing. Take $a_n = 2^{b_n}$ where $b_n = \lfloor \log_2 f(n) \rfloor$ for all $n$, so that $a_n \le f(n)$. Once $b_n \ge k$, the congruence class of $X_n$ modulo $2^{k}$ becomes constant and it is not difficult to deduce that the random walk is transient.
In light of this example, one might ask for the walk to be ``irreducible''.
We say an integer-valued random walk is \emph{irreducible} if, for any $a,b \in\mathbb{Z}$ and any $n \in \mathbb{N}$ such that $\mathbb{P}(X_n = a) > 0$, there is some $m > n$ such that $\mathbb{P}(X_m = b \mid  X_n = a) > 0$.
Clearly, the Rademacher random walk that we just constructed is not irreducible. However, imposing irreducibility does not change the situation, as the following lemma shows.
 \begin{restatable}{lemma}{slow}\label{lem: slowly growing and irreducible but transient}
Let $f: \mathbb{N} \to [1,\infty)$ be any function such that $f(n) \to \infty$ as $n \to \infty$. Then there exists a non-decreasing sequence $(a_n)_{n \ge 1}$ of integer step sizes for which $a_n \le f(n)$ for all $n$ and the associated Rademacher random walk is both irreducible and transient.
\end{restatable}

To address the case where $(a_n)_{n \ge 1}$ is bounded, we prove the following lemma, which extends a theorem of Bhattacharya and Volkov \cite[Theorem 2]{bhattacharya2023recurrence}. 

\begin{lemma}\label{L3}
 Let $\left(a_n\right)_{n \geq 1}$ be a sequence of real numbers such that $\sum_{n=1}^\infty a_n^2 = \infty$, and let $X = (X_n)_{n \ge 0}$ be the associated Rademacher random walk. Then, almost surely, $X$ is unbounded below and unbounded above, and in particular, $X$ changes sign infinitely often.
\end{lemma}
In \cite{bhattacharya2023recurrence} it was shown that $X$ changes sign infinitely often almost surely under the stronger assumption that $\left(a_n\right)$ is a non-decreasing sequence of positive reals. Note that the condition $\sum_{n=1}^\infty a_n^2 = \infty$ is necessary, since if $\sum_{n=1}^\infty a_n^2 < \infty$ then  $X_n$ almost surely converges to a random limit. When $X$ converges to a non-zero limit, it changes sign only finitely many times. Also, $X$  converges to $0$ with positive probability only if $a_n = 0$ for all sufficiently large $n$; this is a consequence of the Erd\H{o}s--Littlewood--Offord anti-concentration inequality, Lemma~\ref{lem:ELO} below. In this case $X$ also does not change sign infinitely often. Hence, $X$ changes sign infinitely often if and only if $\sum_{n=1}^\infty a_n^2 = \infty$.

Lemma~\ref{L3} implies that if the sequence $(a_n)$ is bounded, say $a_n \le C$ for all $n$, and $\sum_{n=1}^\infty a_n^2 = \infty$, then $\mathbb{P}(|X_n| < C \text{ i.o.}) = 1$, so the Rademacher random walk is weakly recurrent.
Recall that if $\sum_{n=1}^\infty a_n^2$ converges, then the Rademacher random walk converges to a random limit and the walk is also weakly recurrent. Putting these two cases together gives the following corollary.
\begin{corollary}\label{C1}
 If the sequence $(a_n)_{n \ge 1}$ is bounded, the associated Rademacher random walk is weakly recurrent.
\end{corollary}
We also have another immediate corollary of Lemma~\ref{L3}:
\begin{corollary}\label{C2}
 If the sequence $(a_n)_{n \ge 1}$ satisfies $\sum_{n=1}^\infty a_n^2 = \infty$ and $a_n \to 0$ as $n \to \infty$, then the associated Rademacher random walk is topologically recurrent.
\end{corollary}
 This is a special case of the well-known fact (which it easily implies by conditioning on the step sizes) that any symmetric random walk with bounded but not necessarily identically distributed steps is weakly recurrent. 
 
Although we know that there is a weakly recurrent Rademacher walk with $a_n = \Theta(n^{1/2})$, we have not  determined the behaviour in the natural case where $a_n = n^\alpha$ for any $\alpha$ in the range $0 < \alpha \le 1/2$. We can, however, say something about the possible behaviours in this range by applying the following theorem. 

\begin{theorem}\label{thm: topological recurrence}
Suppose the sequence $(a_n)$ is unbounded and $a_n - a_{n-1} \to 0$. Then the associated Rademacher random walk is either transient or topologically recurrent.
\end{theorem}
Theorem~\ref{thm:transient} shows that the transient case may occur, and the recurrence result of \cite{bhattacharya2023recurrence} concerning the sequence $a_n = \log n$ shows that the topologically recurrent case may occur.

The rest of the paper is organised as follows. In~\S\ref{SS:anti-concentration} we discuss a number of known anti-concentration results and prove a simple but useful lemma about combining anti-concentration at different scales. In~\S\ref{S:RS anti-concentration} we prove some anti-concentration estimates for Rademacher sums, including a (mod $m$) analogue of the Erd\H{o}s--Littlewood--Offord inequality (~Lemma~\ref{thm: modular ELO}).  In~\S\ref{S:RW anti-concentration} we prove the anti-concentration bound Theorem~\ref{thm:anti-concentration}, from which we have already deduced Theorem~\ref{thm:transient}, and the concentration bound Proposition~\ref{prop:lower-anti}. In~\S\ref{S:weakly recurrent example} we prove Theorem~\ref{thm:recurrent} by exhibiting an explicit sequence $(a_n)_{n \ge 1}$ such that $a_n = \Theta(n^{1/2})$ as $n \to \infty$ and the associated Rademacher walk is weakly recurrent. 
In~\S\ref{S:slowly} we prove Lemma~\ref{lem: slowly growing and irreducible but transient} and Lemma~\ref{L3}.
In~\S\ref{sec:allnat} we prove Theorem~\ref{thm:ints} and deduce Corollary~\ref{cor:logsquared}.
In~\S\ref{sec: topological} we prove Theorem~\ref{thm: topological recurrence}.
Finally, in~\S\ref{sec:fast} we prove Theorem~\ref{thm:fast} and Theorem~\ref{thm:fastincreasing}.

\section{Anti-concentration inequalities}\label{SS:anti-concentration}

In this section we recall several well-known anti-concentration inequalities that will be used later on, and we give a simple lemma which allows us to combine the anti-concentration of two random variables at different scales.

Recall that the \emph{concentration function} $Q$ of a real-valued random variable $A$ is given by 
\[ Q_r(A) = \sup_{x \in \mathbb{R}} \mathbb{P}(x < A \le x+r). \]
 Note that for any integer $m \ge 1$, the union bound gives
\[Q_{mr}(A) \le mQ_r(A).\]

We start with the general case of a one-dimensional Rademacher random walk $X$ whose step sizes are general real numbers, not necessarily integers and not necessarily separated. 
\begin{lemma}[Erd\H{o}s--Littlewood--Offord inequality]\label{lem:ELO} Suppose that the step sizes $(a_n)_{n\ge 1}$ of a Rademacher walk $X = (X_n)_{n \geq 1}$ are all greater than or equal to a constant $c > 0$. Then
\begin{equation} \label{eq: ELO}
  Q_{2c}(X_n) \le \binom{n}{\lfloor n/2 \rfloor}2^{-n} \sim  \sqrt{\frac{2}{\pi n}} \text{ as $n \to \infty$.}
\end{equation}
\end{lemma}
Erd\H{o}s' simple proof of this in \cite{erdos1945lemma} (improving a slightly worse bound by Littlewood and Offord) was to note that for any $x \in \mathbb{R}$, the set of assignments of $(\epsilon_1, \dots, \epsilon_n)$ for which $\sum_{i=1}^n \epsilon_i a_i \in (x,x+c]$ form an anti-chain in the hypercube $\{-1,1\}^n$ with the coordinatewise partial order, and then to apply Sperner's Lemma \cite{sperner1928}, which says that the size of the largest anti-chain is equal to that of the middle layer of the hypercube, i.e.~$\binom{n}{\lfloor n/2\rfloor}$.  In~\S\ref{S:RS anti-concentration} we will prove the following (mod $m$) analogue of Lemma~\ref{lem:ELO} for the case of integer step sizes.
\begin{restatable}{lemma}{modularELO}\label{thm: modular ELO}
Let $m$ be a positive integer and let $b_1, \dots, b_n$ be positive integers coprime to $m$. Let $\epsilon_1, \dots, \epsilon_n$ be independent Rademacher random variables, and let $X_n = \sum_{i=1}^n \epsilon_i b_i$. Then 
\[ \max_{r \in \mathbb{Z}/m\mathbb{Z}} \mathbb{P}(X_n \equiv r \pmod{m}) \le \begin{cases} \frac{1}{m} + \sqrt{\frac{2}{\pi n}} & \text{if $m$ is odd,}\\ \frac{2}{m} + \sqrt{\frac{2}{\pi n}} & \text{if $m$ is even.}\end{cases}\]
\end{restatable}
Lemma~\ref{thm: modular ELO} is a slight improvement of a special case of the following general result of Ju\v{s}kevi\v{c}ius and \v{S}emetulskis \cite[Cor. 1]{juskevicius2017groups} about anti-concentration of random products in arbitrary groups:
\begin{theorem}
Let $G$ be a group with group operation $\ast$ and let $m \ge 2$ be an integer. Let $g_1, \dots, g_n$ be elements of order at least $m$, and consider the $G$-valued random variable $X = g_1^{\epsilon_1} \ast \dots \ast g_n^{\epsilon_n}$, where the $\epsilon_i$ are independent Rademacher$(1/2)$ random variables. Then
\[\sup_{g \in G} \mathbb{P}(X = g) \le \frac{1}{\lceil{m/2}\rceil} + \sqrt{\frac{2}{\pi n}}\,.\]
\end{theorem}
We include a self-contained proof of Lemma~\ref{thm: modular ELO} that is somewhat different from the method of proof in~\cite{juskevicius2017groups}.

The Erd\H{o}s--Littlewood--Offord inequality is a special case of a more general result, known as the Kolmogorov--Rogozin inequality, which gives anti-concentration bounds on a sum of independent random variables using bounds on the anti-concentration of the  summands. The first results of this form were shown by Doeblin and Levy, before being improved by Kolmogorov, then Rogozin and then Kesten. The sharpest possible bounds were obtained recently by Ju\v{s}kevi\v{c}ius in \cite{juskevicius2024sharp}, which summarizes the history of the problem. 

For many applications, the following form of the Kolmogorov--Rogozin inequality suffices.
\begin{theorem}[{Rogozin \cite{rogozin1961concentration}}]\label{thm: Kolomogorov-Rogozin}
    There is a $C > 0$ such that for any independent random variables $X_1, \dots, X_n$ and real numbers $0 < \lambda_1, \dots, \lambda_n \leq 2r$,  
    \[Q_r(X_1 + \dotsb + X_n)\leq C \cdot r \cdot \left( \sum_{i=1}^n \lambda_i^2 (1 - Q_{\lambda_i}(X_i))\right)^{-1/2}.\]
\end{theorem}

Although we will not use it directly, we mention for completeness a key tool in proving many anti-concentration inequalities:
\begin{lemma}[Esseen \cite{esseen1966kolmogorov}]
There is a constant $C$ such that if $X$ is any real-valued random variable, with characteristic function
 $\psi$, then \[Q_r(X) \le C r
\int_{-\pi/r}^{\pi/r} | \psi(\lambda)| \,\textup{d}\lambda.\]
\end{lemma}
A more general version of this lemma, given in Petrov~\cite[Chapter~3, Lemma~3]{Petrov}, states that for every $r \ge 0$ and $a > 0$,
\[Q_r(X) \le \left(\frac{96}{95}\right)^2 \max\left(r, \frac{1}{a}\right)\int_{-a}
^a |\psi(\lambda)| \textup{d}\lambda.\]
Another well-known inequality that serves both as a concentration inequality and as an anti-concentration inequality is the Berry--Esseen inequality, which we state here in Esseen's slightly stronger form.
\begin{theorem}[Esseen \cite{esseen1942}, see also {\cite[Chapter~5, Theorem~3]{Petrov}}]\label{thm:BerryEsseen}
Let $X_1, \dots, X_n$ be independent random variables such that $\mathbb{E}(X_i) = 0$ and $\mathbb{E} (|X_i|^3) < \infty$ for $i = 1, \dots, n$. 
Let \[V_n = \Var\left(\sum_{i=1}^n X_i\right) = \sum_{i=1}^n \mathbb{E} (X_i^2)\]
and consider the cumulative distribution function of the normalized sum,
\[F_n(x) = \mathbb{P} \biggl( V_n^{-1/2} \sum_{i=1}^n X_i < x\biggr).\]
Then, for some universal constant $C$, we have
\[\sup_x |F_n(x) - \Phi(x)| \le C V_n^{-3/2} \sum_{i=1}^n \mathbb{E}(|X_i|^3),\] where $\Phi$ is the cumulative distribution function of a standard Gaussian.
\end{theorem}

Let us now consider the case where the step sizes are separated. Let $Y = (Y_n)_{n\geq 0}$ be the Rademacher random walk with step sizes 1, 2, 3, etc. It is known that for any Rademacher random walk $(X_n)_{n \geq 0}$ with \emph{distinct positive integer} step-sizes, we have the anti-concentration bound
\begin{equation}\label{eq: Stanley and Sullivan}
Q_1(X_n) = \sup_{x \in \mathbb{Z}}\mathbb{P}(X_n = x) \le \mathbb{P}(Y_n = 0) \sim  \frac{\sqrt{6/\pi}}{n^{3/2}} \text{ as } n \to \infty.
\end{equation}
The inequality in~\eqref{eq: Stanley and Sullivan} was proven by Richard Stanley \cite{stanley1980anticoncentration}, using enumerative algebraic geometry and answering a question of Erd\H{o}s and Moser. Another wonderful proof using Lie algebras was given shortly afterwards by Proctor \cite{proctor1982solution}. The asymptotic in~\eqref{eq: Stanley and Sullivan} is due to Sullivan \cite{sullivan2013conjecture}. Before Stanley's result, S\'ark\"ozi and Szemer\'edi \cite{sarkozi1965erdos}
had shown that \[\sup_{x \in \mathbb{Z}} \mathbb{P}(X_n = x) = O(n^{-3/2}).\] From their bound it already follows by a simple Borel--Cantelli argument that any Rademacher random walk whose step sizes are distinct positive integers must be transient (see \cite[Theorem 4]{bhattacharya2023recurrence}). 

Hal\'asz extended the result of S\'ark\"ozi and Szemer\'edi as follows. 
\begin{lemma}[{Hal\'asz \cite[Theorem 2]{halasz1977estimates}}]\label{lem:Halasz}
Consider $n$ vectors $a_1, \dots, a_n \in \mathbb{R}^d$ such that for any unit vector $e$, we have $|\langle e, a_i\rangle| \ge 1$ for at least $\delta n$ values of $i$, and also $\|a_i - a_j\| \ge 1$ whenever $i \neq j$. Then \begin{equation}\label{eqn:Halasz}
\mathbb{P}\left(\sum_{i=1}^n \epsilon_i a_i \in B(x,1)\right) \le  c(\delta,d) n^{-1-d/2}, \end{equation}
where $c(\delta,d)$ is a constant that does not depend on $n$.\end{lemma}
In the case $d=1$, the inner product condition involving $\delta$ may be dropped since it is implied by the separation condition for large enough $n$. We remark that~\eqref{eqn:Halasz} already gives sufficient anti-concentration to show that the Rademacher random walk with step sizes $a_n = n^{\alpha}$ is transient for any $\alpha > 2/3$. Indeed, if $\alpha \geq 1$, then $|a_{i} - a_j| \geq 1$ for all $i \neq j$, and we can immediately apply Lemma~\ref{lem:Halasz} in the case $d=1$ to get an anti-concentration bound of $O(n^{-3/2})$ (and then we can finish by applying the Borel--Cantelli lemma). On the other hand, if $\alpha < 1$, then $a_{n+1} - a_n < 1$ and there is some $a_{i}$ in each interval $[k, k+1)$. Hence, we can find a subsequence $a_{i_1}, a_{i_2}, \dots$ of length $\Theta(n^{\alpha})$ such that $a_{i_j} \in [2j, 2j+1)$, and the theorem gives an anti-concentration bound of $O(n^{-3\alpha/2})$, which suffices to finish using Borel--Cantelli when $\alpha > 2/3$.

All of the bounds above deal with step sizes which are of a somewhat similar scale. If the step sizes grow exponentially with base at least 2, then $Q_1(X_n) = 1/2^n$ and there is much better anti-concentration than given by any of the results above.
While we will not be dealing with scales quite as different as this, combining anti-concentration at different scales and (nearly) multiplying the anti-concentration bounds of each is the key to our proof. For this we use the following lemma.

\begin{lemma}[Combining anti-concentration bounds at different scales] \label{lem: combine scales}{\;}\\
Let $0 < r < s $ and let $A$ and $B$ be independent real-valued random variables. Then
\[Q_{r}(A+B) \le \mathbb{P}(|A| \ge s) + 3 Q_r(A)Q_s(B).\]
Hence,
\[
Q_r(A+B) \le (1-Q_{2s}(A)) + 3Q_r(A)Q_s(B).
\]
\end{lemma}
\begin{proof}
For any $x \in \mathbb{R}$ we have
\begin{equation*}
\begin{split}
\mathbb{P}(x & < A + B \le x +r)  \le  \mathbb{P}(|A| \ge s) + \mathbb{P}(|A| \le s \text{ and } x  < A + B \le x+r)\\
  & \le  \mathbb{P}(|A| \ge s) + \mathbb{P}( x - s < B \le x + r + s) \cdot \sup_{b \in \mathbb{R}}\mathbb{P}(x - b <  A \le x-b +r)\\
& \le  \mathbb{P}(|A| \ge s) + 3 Q_s(B)Q_r(A).
\end{split}
\end{equation*}
Since we may replace $A$ by $A+c$ for any constant $c$ without changing $Q_r(A)$ or $Q_r(A+B)$, the final statement follows.
\end{proof}

\section{\hspace{-6pt}Modular anti-concentration of Rademacher sums} \label{S:RS anti-concentration}
 In the course of several of our proofs, we will need to control the probability that a sum of Rademacher random variables takes a particular value, or lies in a particular congruence class modulo some positive integer. In this section we establish some useful results of these kinds.

 For a real random variable $X$, the \emph{characteristic function} of $X$ is the function $\psi_X:\mathbb{R} \to \mathbb{C}$ given by $\psi_X(t) = \mathbb{E}(e^{itX})$. We say that a real random variable $X$ is \emph{monotone} if $|\psi_X(t)|$ is decreasing on $[0, \pi]$. 
\begin{theorem}[\protect{\cite[Theorem 1.1]{aamand2022sums}}]
    Let $X_1, \dots, X_k$ be independent integer-valued random variables with $\mathbb{E}(X_i) = \mu_i$ and $\Var(X_i) = \sigma^2_i < \infty$. Suppose that their sum $X = X_1 + \dotsb X_k$ is a monotone random variable with mean $\mu$ and variance $\sigma^2$. Then, for every $t$ for which $\mu + t \sigma$ is an integer, we have 
    \[\left| \mathbb{P}(X = \mu + t \sigma) - \frac{1}{\sqrt{2\pi \sigma^2}} e^{-t^2/2} \right| \leq c \left( \frac{\sum_{i=1}^k \mathbb{E}(|X_i - \mu_i|^3)}{\sigma^3}\right)^2,\]
    where $c$ is a universal constant.
\end{theorem}
Note that the sum of independent monotone random variables is monotone and that Bernoulli($1/2$) random variables are monotone.  Hence, by applying the above theorem to a suitably shifted and rescaled sum, we obtain the following local limit theorem for sums of Rademacher random variables.
\begin{corollary}\label{cor:rademacher-local}
    Let $X$ be the sum of $n$ Rademacher random variables. Then, for any $x \equiv n \mod 2$, we have
    \[\left| \mathbb{P}(X=x) - \frac{1}{\sqrt{\pi n /2}} e^{-\frac{x^2}{2n}} \right| \leq \frac{c}{n}, \]
    where $c$ is an absolute constant.
\end{corollary}

In the remainder of this section we establish~Lemma~\ref{thm: modular ELO}, which we restate here.
\modularELO*
Before giving the proof, let us make a few observations about this result.

Fixing arbitrary positive integers $b_1, \dots, b_n$ and letting $m \to \infty$ along the primes gives an alternative proof of the integer case of the Erd\H{o}s--Littlewood--Offord inequality, Lemma~\ref{lem:ELO}.

On the other hand, taking each $b_i =1$ in  gives a simpler modular anti-concentration result that is still useful.
\begin{corollary}\label{cor: max prob on cycle}
Let $m \ge 2$ be an integer and let $X_n = \sum_{i=1}^n \epsilon_i$ be the sum of $n$ independent Rademacher random variables. For every $\eps > 0$,  if $n \ge \eps m^2$, we have
\[ \max_{0 \le r < m} \mathbb{P}(X \equiv r \pmod{m}) < c(\eps)/m,\]
where $c(\eps) = 1 + \sqrt{\frac{2}{\pi\eps}}$.
\end{corollary}

\begin{remark}
For the case where  $n \le m^2$ and the obvious parity condition is satisfied, a complementary lower bound for $\mathbb{P}(X \equiv r \pmod{m})$, also of order $1/m$, may be found in \cite[Corollary 3.2]{bhattacharya2023recurrence}.
\end{remark}

 If $p$ is an odd prime, then a $p$-adic formal analogue of the Erd\H{o}s--Littlewood--Offord inequality follows from~Lemma~\ref{thm: modular ELO}, as follows. If $b_1, \dots, b_n \in \mathbb{Q}_p$ all satisfy $\|b_i\|_p \ge c$, then for any $z \in \mathbb{Q}_p$ we have
 \[ \mathbb{P}(X_n \in B(z,c)) \le \frac{1}{p} + \sqrt{\frac{2}{\pi n}},\]
 where $B(z,c)$ here means the open ball in $\mathbb{Q}_p$ of $p$-adic radius $c$ about $z$.

In the proof of~Lemma~\ref{thm: modular ELO}, we will need the following easy consequence of the rearrangement inequality.
\begin{lemma}\label{lem:max_sum_of_column_products}
Let $n,m \ge 1$ and let $y_1, \dots, y_m$ be given positive numbers. Consider the set $\mathcal{A}$ of $n$ by $m$ matrices in which the elements of each row are some permutation of $(y_1, \dots, y_m)$. 
 Then 
 \[ \max_{A \in \mathcal{A}} \sum_{j=1}^m \prod_{i=1}^n A_{ij} = \sum_{j=1}^m y_j^n.\]
\end{lemma}
\begin{proof}
The set $\mathcal{A}$ is finite, so there exists a matrix $A \in \mathcal{A}$ that maximizes $\sum_{j=1}^m \prod_{i=1}^m A_{ij}$. For this $A$, the terms in each row must be sorted in the same as order as the products over the columns after that row is deleted, by the rearrangement inequality. Hence, they are sorted in the same way as the overall column products, which means that they are all sorted in the same way.
\end{proof}

We now turn to the proof of Lemma~\ref{thm: modular ELO}.
\begin{proof}[Proof of~Lemma~\ref{thm: modular ELO}]
 For each $x \in \mathbb{Z}/m\mathbb{Z}$, let $\delta_x: \mathbb{Z}/m\mathbb{Z} \to \mathbb{C}$ be the indicator function of $x$, and for each $\lambda \in \mathbb{Z}/m\mathbb{Z}$, 
 let $e_\lambda: \mathbb{Z}/m\mathbb{Z} \to \mathbb{C}$
 be the function $e_\lambda(x) = \exp(2\pi i \lambda x/m)/\sqrt{m}$.
  Note that the functions $e_\lambda$ form an orthonormal basis of $\ell^2(\mathbb{Z}/m\mathbb{Z})$. We have
 \[\delta_0= \sum_{\lambda \in \mathbb{Z}/m\mathbb{Z}} \frac{1}{\sqrt{m}} e_\lambda.\]
 A calculation shows that the effect on $e_\lambda$ of convolution with $\frac{1}{2}(\delta_b + \delta_{-b})$ is to multiply it by $\cos (2\pi b\lambda/m)$.
 Hence, the probability mass function of $X_n$ is given by
 \[\mathbb{P}(X_n \equiv r \pmod{m}) = \sum_{\lambda \in \mathbb{Z}/m\mathbb{Z}} \frac{1}{\sqrt{m}}\prod_{i=1}^n \cos(2\pi b_i \lambda/m)  e_\lambda(r)\]
 Since $|e_\lambda(r)| = \frac{1}{\sqrt{m}}$, the triangle inequality gives that
 \begin{equation}\label{eqn: mod anti-conc}
    \mathbb{P}(X_n \equiv r \pmod{m}) \le \frac{1}{m}\sum_{\lambda \in \mathbb{Z}/m\mathbb{Z}} \prod_{i=1}^n |\cos(2\pi b_i \lambda/m)|.
 \end{equation}
 The multi-set of values $(\cos(2\pi b \lambda/m): \lambda \in \mathbb{Z}/m\mathbb{Z})$ is the same for every integer $b$ coprime to $m$. Therefore, by Lemma~\ref{lem:max_sum_of_column_products}, the right-hand side of~\eqref{eqn: mod anti-conc} is maximized when all $b_i$ are congruent (mod $m$) to any constant value $b$ that is coprime to $m$, for example $b=1$. In this case it takes the value $\frac{1}{m}\sum_{\lambda \in \mathbb{Z}/m\mathbb{Z}} |\cos(2\pi\lambda/m)|^n$.
 In the case where $m$ is even,
 \[\frac{1}{m}\sum_{\lambda \in \mathbb{Z}/m\mathbb{Z}} |\cos(2\pi\lambda/m)|^n = \frac{1}{m}\left(2 + 4 \sum_{\ell = 1}^{\lfloor m/4 \rfloor} \cos(2\pi \ell/m)^n\right).\] In the case where $m$ is odd,
 \[\frac{1}{m}\sum_{\lambda \in \mathbb{Z}/m\mathbb{Z}} |\cos(2\pi\lambda/m)|^n = \frac{1}{m}\left(1 + 2 \sum_{\ell = 1}^{\lfloor m/2 \rfloor} \cos(\pi \ell/m)^n\right).\]
 Using the inequality $\cos x \le e^{-x^2/2}$, which holds for $x \in [-\pi/2,\pi/2]$, for even $m$ we obtain
 \begin{align*}
 \frac{1}{m}\sum_{\lambda \in \mathbb{Z}/m\mathbb{Z}} |\cos(2\pi\lambda/m)|^n & \le \frac{1}{m}\left(2 + 4\sum_{\ell=1}^{\infty} e^{-2\pi^2 \ell^2 n/m^2}\right)\\
 & \le \frac{1}{m}\left(2 + 4 \int_0^\infty e^{-(4 \pi^2 n/m^2) \cdot  x^2/2} \,\textup{d}x\right)\\
 & = \frac{2}{m} + \sqrt{\frac{2}{\pi n}}  ,
 \end{align*}
 and similarly for odd $m$ we obtain
 \begin{align*}
 \frac{1}{m}\sum_{\lambda \in \mathbb{Z}/m\mathbb{Z}} |\cos(2\pi\lambda/m)|^n & \le \frac{1}{m}\left(1 + 2\sum_{\ell=1}^{\infty} e^{-\pi^2\ell^2n/(2m^2)}\right)\\
& \le \frac{1}{m} + \sqrt{\frac{2}{\pi n}}.
 \end{align*}
\end{proof}

\section{\hspace{-7pt}Anti-concentration of Rademacher random walks}\label{S:RW anti-concentration}

The main aims of this section are to prove the anti-concentration bound Theorem~\ref{thm:anti-concentration} and the concentration bound Proposition~\ref{prop:lower-anti}.

Our main tool in proving Theorem~\ref{thm:anti-concentration} is Lemma~\ref{lem: combine scales}, for which we need both anti-concentration and concentration bounds on $A$. For the concentration bound, we will use the following well-known bound (see e.g.\ \cite{montgomery1990distribution}), which is a straightforward application of Hoeffding's inequality.

\begin{lemma}\label{lem:rademacher-concentration}
    \[\mathbb{P} \left( \sum_{i=1}^n a_i \epsilon_i \geq t \lVert a \rVert_2 \right) \leq e^{-t^2/2}.\]
\end{lemma}

The result in Theorem~\ref{thm:anti-concentration} only requires the bounds on $a_n$ to hold for large enough $n$, and there is no control over the initial terms. The following result shows that modifying a prefix of length $m$ of the sequence $(a_n)$ can only change the concentration by a factor of at most $2^{m+1}$, and in particular, this will allow us to assume that the bounds on $a_n$ hold for all $n$ when proving Theorem~\ref{thm:anti-concentration}.

\begin{lemma}\label{lem:prefix}
    Let $(a_n)_{n \geq 1}$ be a sequence and let $(a_n')_{n \geq 1}$ be a sequence which differs from $(a_n)$ only in the first $m$ terms. Let $(X_n)_{n \geq 0}$ and $(X'_n)_{n \geq 0}$ be the corresponding Rademacher random walks. Then, for every $n$,
    \[2^{-(m+1)}Q_r(X'_n)\leq Q_r(X_n) \leq 2^{m+1} Q_r(X'_n).\]
\end{lemma}
\begin{proof}
    Suppose that $A$ and $B$ are discrete independent random variables. We claim that
    \begin{equation}
    \label{eq: concentration of A plus B}
    \frac{Q_r(A)Q_r(B)}{2} \leq  Q_r(A+B) \leq Q_r(A).
    \end{equation}
    For the upper bound we have
    \begin{align*}
    \mathbb{P}\left(x < A + B \leq x+r\right) &= \sum_b\mathbb{P}\left(x - b < A \leq x - b + r \right) \mathbb{P}(B = b)\\
    &\leq \sup_b \mathbb{P}\left(x - b < A \leq x - b + r \right) \cdot \sum_b  \mathbb{P}(B = b)\\
    &\leq \sup_x \mathbb{P}\left(x < A \leq x + r \right)\\
     &= Q_r(A),
    \end{align*}
    which implies $Q_r(A+B) \leq Q_r(A)$.
    
    For the lower bound, let $\eps > 0$ be given and choose $x, y$ such that
    \begin{align*}
        \mathbb{P}(x < A \leq x + r) &\ge (1 - \eps) Q_r(A),\\
        \mathbb{P}(y < B \leq y + r) &\ge (1 - \eps) Q_r(B).
    \end{align*}
     Then the events 
     \[\{A+B \in (x+y, x+y+r]\} \quad \text{and}\quad \{A+B \in (x+y + r, x+y + 2r]\}\] 
     are disjoint and cover the event 
     \[\{x < A \leq x+r, \; y < B \leq y + r\}.\] 
     The latter event has probability at least $(1-\eps)^2Q_r(A)Q_r(B)$ and so one of the first two events must have probability at least $(1-\eps)^2Q_r(A)Q_r(B)/2$. Letting $\eps \to 0$ we obtain the lower bound in~\eqref{eq: concentration of A plus B}.

    Now let $A = \sum_{i=1}^m \epsilon_i a_i$ and $B = \sum_{i=m+1}^n\epsilon_i a_i$. Define $A'$ by $A' = \sum_{i=1}^m \epsilon_ia_i'$, so that $X_n = A + B$ and $X_n' = A' + B$.
    Note that $Q_1(A), Q_1(A') \geq 2^{-m}$.
    We have
    \[Q_r(A+B) \leq Q_r(B) \leq 2^{m+1} \frac{Q_r(A') Q_r(B)}{2} \leq 2^{m+1} Q_r(A' + B)\]
    and 
    \[Q_r(A + B) \geq \frac{Q_r(A) Q_r(B)}{2} \geq \frac{Q_r(A) Q_r(A' + B)}{2} \geq 2^{-(m+1)} Q_r(A' + B). \]

\end{proof}

We are now armed with all the tools we need to prove Theorem~\ref{thm:anti-concentration}, which we restate here for convenience.

\anti*

Before we give the full details, let us briefly sketch the proof in the case where $a_n = n^{\alpha}$. The key idea is to group some of the terms in the sum $\sum_{i=1}^n \epsilon_i a_i$ into subsets which are at different scales. In the interval $[2^k, 2^{k+\alpha})$ we expect around $2^{k/\alpha}$ of the $a_n$, and let us set $A$ to be the Rademacher sum of any such terms. Using the Erd\H{o}s--Littlewood--Offord inequality, the anti-concentration of $A$ at the scale $2^{k + 1}$ is around $2^{-k/(2\alpha)}$. On the other hand, if $k$ is much larger than $\alpha$, we expect the magnitude of $A$ to be around $2^{k} \cdot 2^{k/(2\alpha)}$, and there should be concentration around this scale. Hence, if we set $k' = (1 + 1/(2\alpha) + \eps) k$ and let $A'$ be the Rademacher sum of the terms in the interval $[2^{k'}, 2^{k' + \alpha})$, then the sum $A'$ is anti-concentrated at the scale $2^{k'}$ while the sum $A_k$ is concentrated at the scale $2^{k'}$. This means we can use Lemma~\ref{lem: combine scales} to combine the anti-concentration.

This can clearly be repeated and we will take a sequence $k_1, k_2, \dots$ where $k_i = (1 + 1/(2\alpha) + \eps)^i$ and try to combine the anti-concentration of the random variables $A_{1}, A_{2}, \dots$, where we think of $A_i$ as the sum of the terms $\epsilon_i a_i$ for which $a_i \in [2^{k_i}, 2^{k_i+\alpha})$.
In actuality, we will have to widen the interval in which we take the $a_n$ and also limit the number that we take from each interval.

To get the required anti-concentration for all large enough $n$, the anti-concentration from sums of the form $A_1 + \dotsb + A_m$ is not enough. Indeed, while we get the required anti-concentration of $n^{-(\alpha + 1/2 - \gamma)}$ when $n$ is not much bigger than $2^{k_m/\alpha} + 1$, the anti-concentration drops to around $n^{-\alpha}$ by the time $n$ is close to $2^{k_{m+1}/\alpha}$.
In order to apply Lemma~\ref{lem: combine scales} we need $A_1 + \dots + A_{i-1}$ to be concentrated at the scale of $A_i$ for each $i$.
We can achieve this by ensuring that none of the terms in $A_j$ are too big for each $j < m$, but observe that we don't need $A_m$ to be concentrated as it is the last summand.
To get the required concentration for all $n$ between $2^{k_m/\alpha + 1}$ and $2^{k_{m+1}/\alpha + 1}$, we will replace $A_m$ by the sum of the terms $\epsilon_i a_i$ for which $a_i \in [2^{k_m}, \infty)$ and $i \leq n$, which we can do as we do not need concentration for the last summand.

We make this argument rigorous below. We also weaken the conditions on $a_n$. This means we must make our intervals wider to be able to guarantee that we can find enough $a_n$ which fall in the intervals. 
This in turn means that the intervals must be further apart and we pay a penalty in the anti-concentration that we obtain.
However, this is unavoidable; there must be some penalty to pay (for large $\delta$) as shown by Proposition~\ref{prop:lower-anti}. 

\begin{proof}
    Pick a constant $\lambda$ such that $\delta/\alpha < \lambda$. We will later take $\lambda$ to be arbitrarily close to $\delta/\alpha$.

    By Lemma~\ref{lem:prefix}, we can modify the first $m$ terms in the sequence and only change the anti-concentration by at worst a factor of $2^{m+1}$. We will therefore assume that 
    \[cn^{\alpha} \leq a_n \leq Cn^{\alpha + \delta}\]
    holds for all $n \geq 1$ and not just $n$ large enough.
    Using our assumptions on $a_n$, if $n$ is such that 
    \[(2^k/c)^{1/\alpha} \leq n \leq (2^{(1+\lambda)k}/C)^{1/(\alpha + \delta)}
    \]
    then $2^k \leq a_n \leq 2^{(1+\lambda)k}$.    Hence, the number of $n$ for which $a_n$ is in the interval $[2^k, 2^{(1+\lambda)k}]$ is at least
    \[\left(\frac{2^{(1+\lambda)k}}{C}\right)^{\frac{1}{\alpha + \delta}} - 1 - \left(\frac{2^{k}}{c}\right)^{1/\alpha} =  \left(\frac{2^{\frac{\alpha \lambda - \delta}{\alpha (\alpha + \delta)} k}}{C^{1/(\alpha + \delta)}} - \frac{1}{c^{1/\alpha}} \right) 2^{k/\alpha} - 1.\]
    Since we have taken $\alpha \lambda > \delta$, this certainly is at least $2^{k/\alpha}$ for large enough $k$.

    Now fix $\beta \geq \alpha$ and  $\eps > \lambda$, and define $k_i$ by \[k_i = \left(1 + \frac{1}{2\beta} + \eps\right)^{i}.\]
    We will later take $\eps$ arbitrarily close to $\lambda$ and choose the value of $\beta \geq \alpha$ to optimise the anti-concentration bound.
    Let $A_i$ be the random sum corresponding to the first $\lceil 2^{k_i/\beta}\rceil$ of the $a_n$ which lie in the interval $[2^{k_i}, 2^{(1+\lambda)k_i}]$, noting that this is possible when $k_i$ is large by the above argument and our assumption that $\beta \geq \alpha$.

    We now claim that the anti-concentration of the sum of the $A_i$ is roughly equal to the product of their individual anti-concentrations.

    \begin{claim}
        There is $M > 0$ such that 
        \[Q_1(A_1 + \dotsb + A_m) \leq M \cdot 3^m \cdot 2^{-\frac{(1 + 1/(2\beta) + \eps)^{m+1}}{1 + 2\beta \eps}}\]
        for all $m$.
    \end{claim}

    \begin{proof}
    
        First, note that by the Erd\H{o}s--Littlewood--Offord inequality (Lemma~\ref{lem:ELO}), we have 
        \[Q_{2^{k_i + 1}} \left( A_i \right) \leq \sqrt{\frac{2}{\pi}} \cdot 2^{- \frac{k_i}{2\beta}}\]
        for large enough $i$.
        We will induct on $m$, taking $s = 2^{k_{m} + 1}$ and $r = 1$ in  Lemma~\ref{lem: combine scales} to get anti-concentration at different scales.
        
        To apply Lemma~\ref{lem: combine scales}, we first need to show that the sum $A_1 + \dots + A_{m-1}$ is concentrated at the scale $s$.
        Using Lemma~\ref{lem:rademacher-concentration}, we have that
        \begin{align*}
            \mathbb{P}\left(|A_1 + \dotsb + A_{m-1}| \geq s\right) &\leq \sum_{i=1}^{m-1} \mathbb{P}\left(|A_i| \geq \frac{s}{(m-1)} \right)\\
            &\leq 2(m-1) \exp\left( - \frac{s^2}{4 (m-1)^2 2^{(2 + 2\lambda + 1/\beta) k_{m-1}}} \right).
        \end{align*}
        Substituting in $s = 2^{k_{m}  + 1} = 2 \cdot 2^{(1 + \frac{1}{2\beta} + \eps)k_{m-1}}$, we have 
        \[ \mathbb{P}\left(|A_1 + \dotsb + A_{m-1}| \geq  2^{k_{m} + 1}\right) \leq 2(m-1) \exp \left( - \frac{2^{2(\eps - \lambda) k_{m-1}}}{(m-1)^2}\right).\]

        Pick $m_0$ such that for all $m \geq m_0$, the quantity given above is at most
        $2^{-\frac{(1 + 1/(2\beta) + \eps)^{m+1}}{1 + 2\beta \eps}}$, and choose $M \geq 1$ such that 
        \[Q_1(A_1 + \dotsb + A_{m}) \leq M \cdot 3^{m} \cdot 2^{-\frac{(1 + 1/(2\beta) + \eps)^{m+1}}{1 + 2\beta \eps}}\]
        holds for all $m \leq m_0$.

        Now suppose that the claimed bound holds for $m - 1 \geq m_0$. Then, using the above, we have
        \begin{align*}
            Q_1(A_1 + \dots + A_m) &\leq 2^{-\frac{(1 + 1/(2\beta) + \eps)^{m+1}}{1 + 2\beta \eps}} + 3 \cdot M 3^{m-1} 2^{-\frac{(1 + 1/(2\beta) + \eps)^{m}}{1 + 2\beta \eps}} \cdot \sqrt{\frac{2}{\pi}}  2^{- \frac{k_m}{2\beta}}\\
            &= 2^{-\frac{(1 + 1/(2\beta) + \eps)^{m+1}}{1 + 2\beta \eps}} + M \cdot 3^m \cdot \sqrt{\frac{2}{\pi}} \cdot 2^{-\frac{(1 + 1/(2\beta) + \eps)^{m+1}}{1 + 2\beta \eps}}\\
            &= \left(1 + M \cdot 3^m \cdot \sqrt{\frac{2}{\pi}}\right)  2^{-\frac{(1 + 1/(2\beta) + \eps)^{m+1}}{1 + 2\beta \eps}}\\
            &\leq M \cdot 3^m \cdot 2^{-\frac{(1 + 1/(2\beta) + \eps)^{m+1}}{1 + 2\beta \eps}}.
        \end{align*}
        
    \end{proof}

    Pick $N$ and suppose that $2(2^{(1 + \lambda)k_m}/c)^{1/\alpha} \leq N < 2(2^{(1+\lambda)k_{m+1}}/c)^{1/\alpha}$. Let us assume that $N$ is large enough that $m \geq m_0$, where $m_0$ is defined as in the claim above.
    Note that $a_{n} > 2^{(1+\lambda)k_m}$ for all $n \geq N$ so, in particular, the sum $X_{N}$ contains all of the terms in the sum $A_1 + \dotsb + A_{m}$. Let $B$ be the Rademacher sum of the $a_n$ for which $n \leq N$ and $a_n \geq 2^{k_m}$, so that the terms in this sum are a superset of the terms in $A_m$. 
    We will bound the anti-concentration of $A_1 + \dotsb + A_{m-1} + B$ using Lemma~\ref{lem: combine scales}.

    First, note that $B$ contains at least
    \[N - \left(\frac{2^{k_{m}}}{c} \right)^{1/\alpha} \geq \frac{N}{2}\] terms, all of which are at least $2^{k_m}$. By Lemma~\ref{lem:ELO}, we have 
    \[Q_{2^{k_m + 1}}(B) \leq \sqrt{\frac{4}{\pi N}}. \]
    We again have the concentration bound 
    \[ \mathbb{P}\left(|A_1 + \dotsb + A_{m-1}| \geq  2^{k_{m} + 1}\right) \leq 2 (m-1) \exp \left( - \frac{2^{2(\eps - \lambda) k_{m-1}}}{(m-1)^2}\right),\]
    but this time we bound it directly.
    Since $m = \Theta(\log \log (N))$ and this term is
    $2^{-2^{2^{\Theta(m)}}}$, it decays quicker than any polynomial.
    The term $3^m$ is also at most polylogarithmic in $N$, so is certainly in $O(N^{\gamma/2})$ for any fixed $\gamma > 0$.

    We also have that 
    \begin{align*}
        2^{-\frac{(1 + 1/(2\beta) + \eps)^{m}}{1 + 2\beta \eps}} &= 2^{- \frac{(1 + \lambda)(1 + 1/(2\beta) + \eps)^{m + 1}}{\alpha} \cdot \frac{\alpha}{(1 + 2 \beta \eps) (1 + \lambda) (1 + 1/(2\beta) + \eps)}}\\
        &\leq \left(\frac{c^{1/\alpha} N}{2}\right)^{-\frac{\alpha}{(1 + 2 \beta \eps) (1 + \lambda) (1 + 1/(2\beta) + \eps)}}.
    \end{align*}

    Using Lemma~\ref{lem: combine scales}, we find that $Q_1(A_1 + \dotsb + A_{m-1} + B)$ is bounded by 
    \[ M \cdot O(N^{\gamma/2}) \cdot  \left(\frac{c^{1/\alpha} N}{2}\right)^{-\frac{\alpha}{(1 + 2 \beta \eps) (1 + \lambda) (1 + 1/(2\beta) + \eps)}} \cdot \sqrt{\frac{4}{\pi N}} + O\left(2^{-2^{2^{\Theta(\log \log (N))}}}\right)\]
    which is 
    \[O\left(N^{-1/2 -\frac{\alpha}{(1 + 2 \beta \eps) (1 + \lambda) (1 + 1/(2\beta) + \eps)} + \gamma/2} \right) .\]
    
    It remains to choose the values $\lambda$, $\eps$ and $\beta$, and we recall that we require $\alpha \lambda > \delta$, $\eps > \lambda$ and $\beta \geq \alpha$. 
    If $\delta < \frac{\sqrt{\alpha^2 + 1} - \alpha}{2}$, then we take $\beta = 1/(2 \sqrt{\lambda^2 + \lambda})$, else we take $\beta = \alpha$. We then take $\lambda$ sufficiently close to $\delta/\alpha$ and $\eps$ sufficiently close to $\lambda$ so that $\beta \geq \alpha$ and 
    \[\frac{\alpha}{(1 + 2 \beta \eps) (1 + \lambda) (1 + 1/(2\beta) + \eps)} \geq \alpha f(\alpha,\delta) - \frac{\gamma}{2}.\]

\end{proof}

From this it is easy to deduce the following corollary.

\begin{corollary}
    Let $(a_n)_{n \geq 1}$ be a sequence such that $a_n = n^{\alpha + o(1)}$ for some constant $\alpha > 0$, and let $(X_n)_{n \geq 0}$ be the associated Rademacher random walk.
    Then $Q_1(X_n) = O(n^{- (\alpha + 1/2 + o(1))})$.
\end{corollary}

We end this section by proving Proposition~\ref{prop:lower-anti}, which we restate here for convenience.
\lowerAnti*

\begin{proof}
    First, let us consider the sequence $(a_n)_{n \geq 1}$ where $a_n = n^{\alpha}$, and let $(X_n)_{n\geq 1}$ be the associated Rademacher random walk.
    Chebyshev's inequality gives that
    \[\mathbb{P} \left( |X_n| < 2 \sqrt{\Var(X_n)} \right) \geq \frac{3}{4}. \]
 The interval $(-2\sqrt{\Var(X_n)}, 2\sqrt{\Var(X_n)})$ may be covered using no more than  $\lceil4\sqrt{\Var(X_n)}\rceil$ intervals of the form $[x,x+1)$, so there exists $x \in \mathbb{R}$ such that   
    \[\mathbb{P}\left( X_n \in [x, x + 1) \right) \geq \frac{3}{16\lceil \sqrt{\Var(X_n)} \rceil}. \]
    The bound now follows by substituting in
    \[\Var(X_n) = \sum_{k=1}^n (k^{\alpha})^2 = \frac{n^{2\alpha + 1}}{2\alpha  + 1} + O (n^{2\alpha}).\]

    Now set $r = \delta/\alpha$ and consider the sequence where $a_n = 2^{(1 + r)^k}$ for $2^{\frac{(1 + r)^{k-1}}{\alpha}} \leq n < 2^{\frac{(1 + r)^{k}}{\alpha}}$. Note that we have $n^{\alpha} \leq a_n \leq n^{\alpha + \delta}$.

    Let $A_k$ be the Rademacher sum of the $a_n$ which are equal to $2^{(1 + r)^k}$.
    The most likely value for $A_k$ has probability
    \[\binom{N}{\lfloor N/2 \rfloor}2^{-N} \geq \sqrt{\frac{1}{\pi N}}\]
    where $N \geq 2$ is the number of terms equal to $2^{(1 + r)^k}$. We have that $N \leq 2^{\frac{(1 + r)^k}{\alpha}}$, so the most likely value for $A_k$ has probability at least $2^{-\frac{(1 + r)^k}{2\alpha}}/\sqrt{\pi}$.
    Hence, the most likely value for $A_1 + \dots + A_{k}$ has probability at least
    \[\frac{2^{-\frac{((1+r)^{k} - 1)(r+1)}{2r\alpha}}}{\pi^{k/2}} \geq \frac{2^{-\frac{(1+r)^{k+1}}{2r\alpha}}}{{\pi^{k/2}}}.\]
    Hence, if we take $n_k = \lceil 2^{\frac{(1 + r)^{k}}{\alpha}} \rceil - 1$ to the last step with step size $2^{(1 + r)^k}$, then the most likely value has probability at least $n^{-(1/2 + 1/(2r))}/\pi^{k/2}$, as required.
    \end{proof}

\section{\texorpdfstring{A recurrent example with $a_n = \Theta(n^{1/2})$}{A recurrent example where a\textunderscore{}n is close to n\textasciicircum{1/2}}}\label{S:weakly recurrent example}

 Consider the sequence \[(a_n)_{n \ge 1} = (3,1,5,3,5,3,5,3,5,3,9,7,9,7,9,7,9,7,9,7,\dots)\]
which is made up of consecutive blocks, where the $k^{th}$ block has length $4^{k}/2$ and the steps in the $k^{th}$ block alternate between $2^k+1$ and $2^k -1$, starting with $2^k+1$. 
Denote the index of the beginning of the $(k)^{th}$ block by $n_k$. That is,
\[n_1 = 1,\, n_2 = 3,\, n_3 = 11,\, \dots,\, n_k = 1 + \sum_{i=1}^{k-1} 4^i/2 = (4^k + 2)/6.\] 
It is not hard to see that $\sqrt{n/2} \leq a_n \leq 3 \sqrt{n}$ for all $n$, and so $a_n = \Theta(\sqrt{n})$. We will prove that the Rademacher random walk associated with this sequence is weakly recurrent (and hence prove Theorem~\ref{thm:recurrent}). In fact we will show that the walk visits every even integer infinitely often. 

Let $X = (X_n)_{n \geq 0}$ be the associated Rademacher random walk. Note that all the steps of $X$ are odd integers, so to study the return times of $X$ to $0$ we may focus our attention on the random walk $Y$ defined by $Y_n = X_{2n}$, which only visits even integers. We will use hitting probability estimates for the simple symmetric random walk on $\mathbb{Z}^2$ and the Kochen--Stone theorem to show that almost surely $Y$ visits every even integer infinitely often.  The Kochen--Stone theorem has been used before to prove recurrence of random walks; see e.g.~\cite{benjamini1996varying,stack}.

\begin{theorem*}[Kochen and Stone {\cite[Theorem 1]{kochen1964note}}]
Let $Z_1, Z_2, \dots$ be a sequence of random variables, each of which has nonzero mean and positive finite second moment. Suppose in addition that $\limsup_{n \to \infty} (\mathbb{E}(Z_n))^2/\mathbb{E}(Z_n^2) > 0$. Then
\begin{enumerate}
\item[(i)] $\mathbb{P}(\liminf_{n \to \infty} Z_n/\mathbb{E}(Z_n) \le 1) > 0$,
\item[(ii)] $\mathbb{P}(\limsup_{n \to \infty} Z_n/\mathbb{E}(Z_n) \ge 1) > 0$, and
\item[(iii)] $\mathbb{P}(\limsup_{n \to \infty} Z_n/\mathbb{E}(Z_n) > 0) \ge \limsup_{n \to \infty} (\mathbb{E}(Z_n))^2/\mathbb{E}(Z_n^2)$.
\end{enumerate}
\end{theorem*}

Let $E_k$ be the event that $X$ visits $0$ during the $(2k)^{th}$ block. We will show the following:
\begin{lemma}\label{lem: hitting estimates}
For $E_k$ as defined above, we have $\mathbb{P}(E_k) = \Omega(1/k)$ as $k \to \infty$, so that $\sum_{k=1}^\infty \mathbb{P}(E_k) = \infty$.
Moreover, there is a finite constant $C$ such that for all $j < k$ we have
\[\mathbb{P}(E_k \mid E_j) \le C \mathbb{P}(E_k).\]
\end{lemma}

Before proving Lemma~\ref{lem: hitting estimates}, let us explain how it implies our claim about the recurrence of the walk $Y$. Let $Z_n = \sum_{k=1}^n \mathbbm{1}(E_k)$. From the first statement in Lemma~\ref{lem: hitting estimates} we deduce that $\mathbb{E}(Z_n) = \Omega(\log n)$, and in particular $\mathbb{E}(Z_n) \to \infty$ as $n \to \infty$. 
From the second statement in Lemma~\ref{lem: hitting estimates} we obtain
for all $j \neq k$ that \[\mathbb{P}(E_j \cap E_k) \le C \,\mathbb{P}(E_j)\mathbb{P}(E_k),\] so
 \begin{align*}
  \mathbb{E}[Z_n^2] &=  \mathbb{E}\left(\left(\sum_{k=1}^n \ind_{E_k}\right)^2\right)\\
 &=  \sum_{k=1}^n \mathbb{P}(E_k) + \sum_{j=1}^n\sum_{k=1}^n\ind_{\{j \neq k\}} \mathbb{P}(E_j \cap E_k)\\
 &\le  \mathbb{E}(Z_n) + C \sum_{j=1}^n \sum_{k=1}^n \ind_{\{j \neq k\}} \mathbb{P}(E_j)\mathbb{P}(E_k)\\
 &\leq \mathbb{E}(Z_n) + C\bigl(\sum_{j=1}^n \mathbb{P}(E_j)\bigr)^2\\
 &= \mathbb{E}(Z_n) + C(\mathbb{E}(Z_n))^2
 \end{align*}
 Since $\mathbb{E}(Z_n) \to \infty$ as $n \to \infty$ we obtain
\[ \frac{(\mathbb{E}(Z_n))^2}{\mathbb{E}(Z_n^2)} \ge  \frac{1}{C}-o(1).\]
This allows us to apply part (iii) of the Kochen--Stone theorem to deduce that 
\[\mathbb{P}\left(\limsup_{n \to \infty} \frac{Z_n}{\mathbb{E}(Z_n)} > 0\right) \ge \frac{1}{C}.\]
Since $\mathbb{E}(Z_n) \to \infty$ as $n \to \infty$, this implies $\mathbb{P}(E_k \text{ occurs i.o.}) > 0$ and in particular $\mathbb{P}(X_{2n} = 0 \text{ i.o.}) > 0$. It then follows from Kolmogorov's zero-one law that 
\[\mathbb{P}(\exists r \in \mathbb{Z}\,:\, Y \text{ visits $r$ i.o.}) = 1,\]
since this is a tail event which occurs with positive probability.

From this it is easy to see that the probability that $Y$ (and therefore $X$) visits every even integer is 1. Indeed, for every time $t$ we have 
\[\mathbb{P}(Y_{t+1} = Y_t + 2 \mid Y_t) = 1/4 = \mathbb{P}(Y_{t+1} = Y_t - 2 \mid Y_t),\]
and hence, for all $m, n \in 2\mathbb{Z}$,
\[\mathbb{P}(\text{$Y$ visits $m$ i.o.\ but does not visit $n$ i.o.}) = 0.\]
The result now follows by summing over the choices for $m$.

 We are now ready to give a proof of Lemma~\ref{lem: hitting estimates}.

\begin{proof}[Proof of Lemma~\ref{lem: hitting estimates}]
For any event $E$ such that $\mathbb{P}(E) > 0$ and any integrable random variable $T \ge 0$ on the same probability space such that $\mathbb{E}(T) > 0$ and $T =0$ on the complement of $E$, we have
\[ \mathbb{P}(E) = \frac{\mathbb{E}(T \mathbbm{1}_E)}{\mathbb{E}(T \mid E)} = \frac{\mathbb{E}(T)}{\mathbb{E}(T \mid E)}.\]
Take $E = E_k$, the event that $X$ visits $0$ during the $(2k)^{th}$ block, i.e.\ during the times $(n_{2k}, \dots, n_{2k+1}-1)$, and let $T = T_k$, where $T_k$ is the number of visits of $X$ to $0$ during the $(2k)^{th}$ block.  

We will prove the following four estimates, where $c_1$ and $c_2$ are positive constants:
\begin{enumerate}
\item[(i)]  $\mathbb{E}(T_k) = \Omega(1)$ as $k \to \infty$,
\item[(ii)]  $\mathbb{E}(T_k \mid E_k) = O(k)$ as $k \to \infty$,
\item[(iii)] $\mathbb{E}(T_k \mid E_j) < c_1$, for all $j, k$ such that $j < k$,
\item[(iv)]  $\mathbb{E}(T_k\mid E_j \cap E_k) > c_2 k$ for all $j, k$ such that $j < k$.
\end{enumerate}
Observe that (i) and (ii) imply that
\[\mathbb{P}(E_k) = \frac{\mathbb{E}(T_k)}{\mathbb{E}(T_k \mid E_k)} = \Omega(1/k) \text{ as $k \to \infty$.}\] 
On the other hand, (iii) and (iv) imply in the same way that
\[\mathbb{P}(E_k \mid E_j) =  \frac{\mathbb{E}(T_k \mid E_j)}{\mathbb{E}(T_k \mid E_j \cap E_k)} < \frac{c_1}{c_2 k} \text{ for all $j, k$ with $j < k$.}\] 
It follows that there is a finite constant $C$ such that for all $j < k$
\[ \mathbb{P}(E_k \mid E_j) \le C \mathbb{P}(E_k).\]

We make the estimates (i)-(iv) by relating the portion of the walk $Y$ during the $(2k)^{th}$ block to a simple symmetric random walk on $\mathbb{Z}^2$. Fix $k$ for now. 
Let $m_0 = (n_{2k}-1)/2$ and $m_1 = (n_{2k+1}-1)/2$. Note that 
\[m_1 - m_0 = \frac{n_{2k+1} - n_{2k}}{2}  = 4^{2k}/2 = 2^{4k-2}.\]
We define a walk $(a_m,b_m)$ for $m_0 \le m \le m_1$ by setting  
$(a_{m_0},b_{m_0}) = (0,0)$, and then, inductively for $m = m_0 + 1, \dots, m_1$, 
\[ (a_m, b_m) = (a_{m-1},b_{m-1}) + \begin{cases} (1,0)  & \text{ if $Y_{m} - Y_{m-1} = 2^{2k+1}$,} \\  (-1,0)  & \text{ if $Y_{m} - Y_{m-1} = - 2^{2k+1}$,} \\  (0,1)  & \text{ if $Y_{m} - Y_{m-1} = 2$,} \\  (0,-1)  & \text{ if $Y_{m} - Y_{m-1} = -2$.}\end{cases} \]
Let $L$ be the random arithmetic progression in $\mathbb{Z}^2$ defined by
\[L = \{(a,b) \in \mathbb{Z}^2 \,:\, 2^{2k+1} a + 2b + Y_{m_0} = 0\}. \]
For $m_0 \le m \le m_1$ we have $Y_m = 0$ if and only if $(a_m,b_m) \in L$. Thus, $T_k$ is the number of times that $(a_m,b_m)$ hits $L$. 

\paragraph*{Proof of (i)} To get a lower bound on $\mathbb{E}(T_k)$, we first apply Chebyshev's inequality to show that $Y_{m_0}$ is often not too large:
\begin{align*}
    \mathrm{Var}(Y_{m_0}) = \mathrm{Var}(X_{n_{2k -1}}) &= \sum_{n=1}^{n_{2k}-1} a_n^2\\
    &= \sum_{n=1}^{2k-1} \frac{4^k}{4} ((2^{n} + 1)^2 + (2^{n} - 1)^2) \\
    &\le 4^{2k}/12 \cdot 2 (1+2^{2k})^2\\
    &\le 2^{8k}.
\end{align*}
Hence, using that $\mathbb{E}(Y_{m_0}) = 0$, we have that
\[\mathbb{P}(|Y_{m_0}| < 2^{1+4k}) \ge 1/4.\]
When $|Y_{m_0}| < 2^{1+4k}$,
 we can express $Y_{m_0}$ as $2^{2k+1}\alpha + 2\beta$,
  where $|\alpha| \le 2^{2k}$ and $|\beta| \le 2^{2k}$. The expected number of visits of $(a_m, b_m)$ to $(-\alpha, -\beta)$ is 
  \[ \sum_{m = m_0+1}^{m_1} \mathbb{P}((a_m,b_m) = (-\alpha,-\beta)).\]
We now use the well-known observation about the simple symmetric two-dimensional random walk that $(a_t + b_t)_{t = m_0}^{m_1}$ and $(a_t - b_t)_{t = m_0}^{m_1}$ are independent simple symmetric random walks on $\mathbb{Z}$, started at $0$ at $t=m_0$. By Corollary~\ref{cor:rademacher-local}, whenever $m -m_0 > (m_1-m_0)/2 =  2^{4k-3}$ and $m-m_0$ has the same parity as $\alpha+\beta$, we have
  \begin{align*} \mathbb{P}((a_m,b_m)  =  (-\alpha,-\beta)) &= \mathbb{P}(a_m + b_m = -\alpha-\beta)\mathbb{P}(a_m - b_m = -\alpha + \beta) \\
  &\geq \left(\frac{1}{\sqrt{\pi (m-m_0)/2}} e^{-\frac{(\alpha + \beta)^2}{2 (m-m_0)}} - \frac{c}{m - m_0}\right)\\ &\quad\quad\cdot \left(\frac{1}{\sqrt{\pi (m-m_0)/2}} e^{-\frac{(\alpha - \beta)^2}{2 (m-m_0)}} - \frac{c}{m - m_0}\right)\\
  &\geq \left(\frac{e^{-16}}{\sqrt{\pi 2^{4k-3}}} - \frac{c}{2^{4k-3}} \right)^2\\
  &= \Omega(2^{-4k}).
  \end{align*}
  Since the number of values of $m$ to which this applies is  $2^{4k-4} $, we obtain the asymptotic lower bound (i), i.e.~$\mathbb{E}(T_k) = \Omega(1)$ as $k \to \infty$.

\paragraph*{Proof of (ii)} By conditioning on the time of the first visit to $0$ during the $(2k)^{th}$ block, we find that \[\mathbb{E}(T_k \mid  E_k) \le \mathbb{E}(T_k \mid X_{n_{2k}} = 0).\] Hence, to prove (ii) it suffices to show
 that $\mathbb{E}(T_k \mid X_{n_{2k}} = 0) = O(k)$ as $k \to \infty$.
 
 The expected number of returns to $(0,0)$ in the first $2N$ steps of simple symmetric random walk on $\mathbb{Z}^2$ started at $(0,0)$ is
 \[\sum_{i = 1}^{N}2^{-2i}\binom{2i}{i}^2 =
 \Theta\left( \log(N)\right).\]
 Hence, the expected number of visits of $(a_t,b_t)_{t=m_0}^{m_1}$ to $(0,0)$ is $\Theta(\log(\frac{m_1-m_0}{2})) = \Theta(k)$. We now show that the expected number of visits of $(a_t,b_t)_{t=m_0}^{m_1}$ to all the nonzero points in the arithmetic progression $L_0$ is $O(1)$, where
 \[L_0 = \{(a,b) \in \mathbb{Z}^2\,:\, 2^{2k+1}a + 2b = 0\} = \langle (1,-2^{2k}) \rangle.\]
 \begin{align*}
 \sum_{m = m_0+1}^{m_1}\mathbb{P}&((a_{m},b_{m}) = (n,-2^{2k}n))\\ 
 &= \sum_{m = m_0+1}^{m_1}\mathbb{P}(a_{m} + b_{m}  = n(1-2^{2k})) \mathbb{P}(a_{m} - b_{m}  = n(1+2^{2k}) ) \\
 &\le 2^{4k-2} \sup_{m \geq 2^{4k}n} \mathbb{P}(a_{m} + b_{m}  = n(1-2^{2k})) \mathbb{P}(a_{m} - b_{m}  = n(1+2^{2k}) ) \\
 &\leq 2^{4k-2} \sup_{m \geq 2^{4k}n} \left( \frac{1}{\sqrt{\pi m/2}} e^{- \frac{(n(2^{2k} - 1))^2}{2m}} + \frac{c}{m}\right)^2\\
 &\leq 2^{4k-2} \left( \sqrt{\frac{2}{\pi e (n(2^{2k} - 1))^2}} + \frac{c}{2^{4k}n}\right)^2\\
 & \le c/n^2,
 \end{align*}
Summing over all $n \in \mathbb{Z}\setminus\{0\}$ gives an upper bound of $c\pi^2/3$, and so we have proved estimate (ii).

\paragraph*{Proof of (iii)}
For estimate (iii), we have
\begin{align*}
\mathbb{E}(T_k \mid E_j) &\le \max_{r \in 2\mathbb{Z}} \mathbb{E}(T_k \mid X_{n_{2j+1}} = r) \\
&\le \sum_{n = n_{2k}}^{n_{2k+1}} Q_1(X_n - X_{n_{2j+1}})\\
&\leq \frac{4^{2k}}{2} \cdot  Q_1(X_{n_{2k}} - X_{n_{2k-1}}).
\end{align*}
Let $m_k$ be the number of pairs $(2i, 2i+1)$ with $n_{2k-1} \le 2i < n_{2k}$ for which  $\epsilon_{2i} = - \epsilon_{2i+1}$. Then $m_k$ is a binomial random variable with $4^{2k}/{4}$ trials and success probability $1/2$. A Chernoff bound immediately implies that the probability that $m_k$ lies in the interval $[4^{2k}/12, 4^{2k}/6]$ is $1 - o(4^{-2k})$.
Conditional on the value of $m_k$, the increment $X_{n_{2k}} - X_{n_{2k-1}}$ is expressible as a sum $2 A_1 + 2^{2k+1} A_2$, where $A_1$ is a sum of $m_k$ Rademacher random variables, $A_2$ is a sum of $4^{2k}/4 - m_k$ Rademacher random variables, and $A_1$ and $A_2$ are independent.

Now condition on $m_k$, and assume that $m_k \in [4^{2k}/12, 4^{2k}/6]$. Consider the digits of $A_1 + 4^kA_2$ in base $4^k$.
The units digit is determined by $A_1$ alone, and by Corollary~\ref{cor: max prob on cycle} there is a constant $c$ not depending on $k$ or on $m_k$ such that $\sup_{r}\mathbb{P}(A_1 \equiv r\pmod{4^k}) \le c/4^k$.

Conditional on $m_k$ and $A_1$, the next digit is determined by $A_2$ and again $\sup_{s}\mathbb{P}(A_2 \equiv s \pmod{4^k}) \le c/4^k$. Hence, each possible value of the last two digits occurs with probability at most $c^2/4^{2k}$. It follows that
\[ Q_1(X_{n_{2k}}-X_{n_{2k-1}}) \le c^2/4^{2k} + o(4^{-2k}).\]
This proves estimate (iii).

\paragraph*{Proof of (iv)}
For estimate (iv), we begin by noting that the arguments used for estimates (i) and (ii) still work when we further condition on $E_j$. Indeed, the estimate for (ii) is unchanged (as we immediately condition on $X_{n_{2k}} = 0$) and the  estimate for (i) is only improved by conditioning on $E_j$ (as this reduces the variance of $Y_{m_0}$).
Hence, we have 
\[\mathbb{P}(E_k \mid E_j) \ge c/k,\] for a constant $c > 0$ that does not depend on $j$.  Let $A_k$ be the event that $X$ visits $0$ between times $n_{2k}$ and $n_{2k+1} - 4^k$. Let $T'_k$ be the number of visits during this interval, so $A_k = \{ T'_k \ge 1\}$ and $E_k \setminus A_k \subseteq \{T_k - T'_k \ge 1\}$. This implies that $\mathbb{P}(A_k^c \mid E_k \cap E_j) \le \mathbb{E}(T_k - T'_k \mid E_k \cap E_j)$, and hence
\[\mathbb{E}(T_k - T'_k \mid E_k \cap E_j) \le \frac{\mathbb{E}(T_k - T'_k\mid E_j)}{\mathbb{P}(E_k \mid E_j)} \le \frac{k}{c} \mathbb{E}(T_k - T_k' \mid E_j).\] 

To estimate $\mathbb{E}(T_k - T_k' | E_j)$ we repeat the method that we used for estimate (iii). Recall that $Q_1(X_n - X_{n_{2j + 1}}) \leq Q_1 (X_{n_{2k}} - X_{n_{2k -1}}) = O(4^{-2k})$, and so
\begin{align*}
\mathbb{E}(T_k - T'_k \mid E_j) &\le \sum_{n = n_{2k+1} - 4^k}^{n_{2k+1}} \mathbb{P}(X_n = 0 \mid E_j)\\
&\le \sum_{n = n_{2k+1} - 4^k}^{n_{2k+1}} Q_1(X_n - X_{n_{2j+1}})\\
&= O(4^{-k}).
\end{align*}

Hence, \[\mathbb{P}(A_k | E_k \cap E_j) = 1 - O(k4^{-k}) = 1 - o(1) \quad \text{ as $k \to \infty$,}\]
uniformly in $j$. Now
\begin{align*}
\mathbb{E}(T_k \mid E_k \cap E_j) &\ge \mathbb{P}(A_k \mid E_k \cap E_j)\, \mathbb{E}(T_k \mid A_k \cap E_j)\\
&\ge (1 - o(1)) \min_{n \in \{n_{2k}, \dots, n_{2k+1}-4^k\}}\mathbb{E}(T_k \mid X_n = 0)\\
&= \Omega(\log(4^k))\\
&= \Omega(k),
\end{align*}
where we have again used that the expected number of returns to $0$ of a two-dimensional simple symmetric random walk in its first $2N$ steps is $\Omega(\log(N))$.
\end{proof}

\section{Recurrence and transience for slowly growing step sizes}\label{S:slowly}

We start by showing that a slowly growing non-decreasing integer sequence $(a_n)_{n \ge 1}$ gives a transient Rademacher random walk if the set of values it takes is a little sparse.

\begin{lemma}\label{lem:sparsevalues} Suppose $(a_n)_{n \ge 1}$ is a non-decreasing sequence of positive integers which takes values in a set $S$. Suppose that the value $s$ appears $L_s$ times in $(a_n)$ and suppose further that $\sum_{s \in S} 1/s < \infty$ and, for some $\eps > 0$ and all large enough $s$, there is some $s' < s$ for which $s'$ is coprime to $s$ and $L_{s'} \geq \eps s^2$.
Then the Rademacher random walk $X$ associated to $(a_n)_{n \ge 1}$ is transient.
\end{lemma}
\begin{proof}
Fix any finite set $F$.
Using the hypothesis about $s'$, we can apply Corollary~\ref{cor: max prob on cycle} to see that for all large enough $s \in S$, the probability that $X_n$ is congruent to any element of $F$ (mod $s$) at the beginning of the block of steps of size $s$ is at most $|F|c(\eps)/s$. The walk $X$ can only visit $F$ during the $s$-block if it is congruent to an element of $F$ modulo $s$ and, since $\sum_{s \in S} 1/s < \infty$, the first Borel--Cantelli lemma shows that almost surely this happens for only finitely many $s \in S$. Hence, $X$ is transient.
\end{proof}

Now we can easily prove~Lemma~\ref{lem: slowly growing and irreducible but transient}, which we restate here for convenience.
\slow*
\begin{proof}
Since we could choose to start with any finite number of steps with a step size of $0$, we may assume without loss of generality that $f(n) \geq 9$ for all $n$. 
Let $p_i = (2i+1)^2$ for each $i \ge 1$, and choose a sequence $\ell_1, \ell_2, \dots$ of integers such that for each $i \ge 1$ we have
\begin{enumerate}
    \item $\ell_{2i} \equiv i+1 \pmod{2}$,
    \item $\ell_{2i+1} \equiv i \pmod 2$,
    \item  $\ell_i \geq p_{i+1}^2$, 
    \item $f(n) \geq p_{i+1}$ for all $n > \sum_{k=1}^{i}\ell_k$.
\end{enumerate}
This is always possible as $f(n) \to \infty$ as $n \to \infty$.
Construct the sequence $(a_n)_{n \ge 1}$ by letting the first $\ell_1$ terms be $p_1$, the next $\ell_2$ terms be $p_2$, and so on. That is, $a_n = p_i$ whenever $ c_i < n \le c_{i+1}$, where $c_i = \sum_{j=1}^{i-1} \ell_j$. The final condition in the list above ensures that $a_n \le f(n)$ for all $n$. Let $(X_n)_{n \ge 0}$ be a Rademacher walk with step sizes given by the sequence $(a_n)_{n \ge 1}$. Since $\gcd(p_{i-1},p_i) = 1$, the construction ensures that  $(a_n)_{n \ge 1}$ satisfies the hypotheses of~Lemma~\ref{lem:sparsevalues}, so $X$ is transient.

It remains to show that $X$ is irreducible. Note that for any $i \ge 1$ we have $(i+1) p_{2i} - i p_{2i+1} = 1$. By our assumptions on the parity of the $\ell_i$, we have that in two consecutive blocks where the step sizes are $p_{2i}$ in the first and $p_{2i+1}$ in the second, it occurs with positive probability that the total increment in block $2i$ is $(i+1)p_{2i}$ and the total increment in block $(2i+1)$ is $-i p_{2i+1}$, in which case the total increment from these two blocks is $1$. Likewise, it occurs with positive probability that the total increment in block $i$ is $-(i+1)p_{2i}$ and the total increment in block $(i+1)$ is $ip_{2i+1}$, so that the total increment from the two blocks is $-1$. Hence, $X$ is irreducible.
\end{proof}

We now turn to the case of bounded step sizes, and prove Lemma~\ref{L3}, which states that if $\sum_{n=1}^\infty a_n^2 = \infty$ then the Rademacher random walk $X$ associated to $(a_n)_{n \ge 1}$ is almost surely unbounded both above and  below.

\begin{proof}[Proof of Lemma~\ref{L3}]
To show that $(X_n)$ is almost surely unbounded both below and above, it suffices to show for any constant $C$ that almost surely $X_n \le C$ i.o.\ and $X_n \ge C$ i.o.\ as well. To prove this, we will show that whenever $\mathbb{P}(X_m = x) > 0$, we have
\begin{equation}
\label{eqn:sign change eventually}
\mathbb{P}(\exists n > m\,:\,(X_n - C)(x-C) \le 0 \mid X_m = x)  = 1.
\end{equation}

For any $n > m$, the increment $X_n - X_m$ is independent of $X_m$ and has a symmetric distribution with variance $\sum_{k=m+1}^{n} a_k^2$. Khintchine's inequalities state that the $L^p$ norm of a Rademacher sum is comparable with the $\ell^2$ norm of its coefficient sequence. In particular, the (sharp) bound for the case $p=4$ is easy to derive:
\begin{align*}
\mathbb{E}((X_n - X_m)^4) & = \mathbb{E}\left(\left(\sum_{k=m+1}^{n}\epsilon_k a_k\right)^4\right)\\
&=  \sum_{k=m+1}^{n} a_k^4  + 3\sum_{k=m+1}^{n}\sum_{j=m+1}^{n}\ind_{\{j \neq k\}} a_k^2 a_j^2\\
&=  3 \sum_{k=m+1}^{n}\sum_{j=m+1}^{n} a_k^2 a_j^2  - 2\sum_{k=m+1}^{n} a_k^4\\
&\le 3 \left(\sum_{k=m+1}^{n} a_k^2\right)^2 .
\end{align*}

We now apply the Paley--Zygmund inequality to the random variable $Z$ defined by $Z = (X_n - X_m)^2$.
\begin{align*}
\mathbb{P}\left(|X_n - X_m| \ge \frac{1}{2}\left(\sum_{k=m+1}^{n} a_k^2\right)^{1/2}\right) &= \mathbb{P}\left(Z > \frac{1}{4}\mathbb{E}(Z)\right) \\ 
&\ge \left(\frac{3}{4}\right)^2 \frac{\mathbb{E}(Z)^2}{\mathbb{E}(Z^2)}\\
&\ge \frac{3}{16}. \end{align*}
We remark that the above inequality complements the statement of Tomaszewski's conjecture, recently proved by Keller and Klein \cite{keller2022proof}, which says that \[\mathbb{P}\left(|X_n - X_m| \le \left(\sum_{k=m+1}^{n} a_k^2\right)^{1/2}\right) \ge \frac{1}{2}.\]

We can now prove~\eqref{eqn:sign change eventually}.  Define a sequence of stopping times $\tau_0 = m < \tau_1 < \tau_2 < \dots$ inductively by
\[\tau_i =  \min\left(\left\{n > \tau_{i-1}\,:\, \frac{1}{2}\left(\sum_{k=1+\tau_{i-1}}^n a_k^2\right)^{1/2} > \left|X_{\tau_{i-1}}-C\right| \right\}\right).\]
Since $\sum_{k=1}^\infty a_k^2 = \infty$, we have $\tau_i < \infty$ a.s.~for every integer $i \ge 0$. Let $\mathcal{G}_i$ denote the $\sigma$-algebra generated by $X_1, \dots, X_{\tau_i}$. For each $i \ge 1$ we have
\[ \mathbb{P}((X_{\tau_i}-C)(X_{\tau_{i-1}}-C) \le 0 \mid \mathcal{G}_{i-1}) \ge \frac{3}{32}.\]
By the conditional Borel--Cantelli lemma, we find that almost surely there exists a random $i < \infty$ such that $(X_{\tau_i}-C)(X_{\tau_{i-1}}-C) \le 0$ and,  taking $n = \tau_i$ for the least such $i$, we have $(X_n-C)(X_m-C) \le 0$.
\end{proof}

\section{Transience for sequences that cover all natural numbers}
\label{sec:allnat}

The main aim of this section is to prove Theorem~\ref{thm:ints}, and then to deduce Corollary~\ref{cor:logsquared}.

\incInts*

\begin{proof}
    Fix $C > 0$, and let $E_n$ be the event that the walk is within $C$ of the origin after one of the steps of size $n$, that is, $E_n = \{ \exists j : a_j = n, |X_j| \le C\}$. We will show that $\sum_n \mathbb{P}(E_n) < \infty$ and therefore almost surely only finitely many of the $E_n$ occur. This means that the probability that the walk is $C$-recurrent is zero for all $C$, and the walk is transient.

    Clearly, the probability of $E_n$ is $0$ if $L_n = 0$, so suppose that $L_n \geq 1$ and that $n$ is large enough for equations~\eqref{eqn:assump1} and~\eqref{eqn:assump2} to hold.
    Let $N = N(n) = \sum_{m=1}^{n-1} L_m$ so that $a_{N+1} = \dotsb = a_{N + L_n} = n$. 
    We split the event $E_n$ into two cases based on the size of $|X_N|$. 
    When $|X_N|$ is large, the probability that $L_n$ steps of size $n$ will travel far enough to be within $C$ of the origin is small enough to be summable. When $|X_N|$ is small, we use the fact that $X_N$ is well-distributed over the equivalence classes modulo $n$, so the probability that steps of size $n$ could possibly get within $C$ of the origin is $O(C/n)$. By combining this with the probability that $|X_N|$ is small, we will find that $\sum_{n=1}^\infty \mathbb{P}(E_n) < \infty$.

 Define \[M = \sum_{m=1}^{n-1} m^2 L_m/\log^{2+\eps}(n).\]
    First, we consider the case where $|X_N|$ is large, by which we mean $|X_M| > 2\sqrt{M}$. By assumption~\eqref{eqn:assump2}, we have $M \ge n^4 \log n$. We may assume that $n$ is large enough  that this implies $\sqrt{M} \ge C$. Also by~\eqref{eqn:assump2} we have $M \ge 4 n^2 L_n \log n$. By the reflection principle and the Azuma--Hoeffding inequality (see Lemma~\ref{lem:rademacher-concentration}) we have
    \begin{align*}
        \mathbb{P}\left( \max_{0 \leq t \leq L_n} \sum_{i=1}^t \epsilon_{(N+i)} n > \sqrt{M}\right) &= 2 \mathbb{P} \left(\sum_{i=1}^{L_n }\epsilon_i n \geq \sqrt{M} \right) - \mathbb{P} \left(\sum_{i=1}^{L_n }\epsilon_i n = \sqrt{M} \right)\\
        &\leq 2 \exp\left({-\frac{M}{2n^2L_n}}\right)\\
        &\leq 2\exp(-2\log n)\\
        &= \frac{2}{n^2}.
    \end{align*}
  By symmetry we also have
  \[\mathbb{P}\bigg( \min_{0 \leq t \leq L_n} \sum_{i=1}^t \epsilon_{(N+i)} n < -\sqrt{M}\bigg) \le \frac{2}{n^2}.\]
Hence,
\[\mathbb{P}\big(E_n \mid |X_N| > 2\sqrt{M}\big) \le \frac{4}{n^2},\]
so
\begin{equation} \label{eq:largebound} \mathbb{P}\big(E_n \cap \{ |X_N| > 2\sqrt{M}\}\big) \le \frac{4}{n^2}.\end{equation}
    
    Now consider the case where $|X_N| \leq 2\sqrt{M}$. We split the sum $X_N$ into two parts as 
    $X_N = S_A + S_B$ where $A \cup B = \{1, \dots, N\}$, $A\cap B = \emptyset$, and 
    \[S_A = \sum_{i \in A} a_i \epsilon_i, \quad S_B = \sum_{i \in B} a_i \epsilon_i.\]
    The idea is that $A$ is small enough that $|S_A|$ is moderately unlikely to be larger than $2\sqrt{M}$, but large enough to ensure that $X_N$ is well distributed over the equivalence classes modulo $n$. On the other hand, the Berry-Esseen inequality will show that $|S_B|$ is moderately unlikely to be smaller than $4\sqrt{M}$, so by a union bound $|X_N|$ is moderately unlikely to be smaller than $2\sqrt{M}$.
    
    By equation~\eqref{eqn:assump1}, there are at least $2n^2$ indices $i \in \{1, \dots, N\}$ such that $\gcd(a_i, n) = 1$. Let $A$ be the set consisting of the first $n^2$ such indices and let $B = \{1, \dots, N\} \setminus A$, so $|B| \ge n^2$ and $\sum_{m \in B} a_m^2 \ge \sum_{m \in A} a_m^2$, and hence by~\eqref{eqn:assump2} we have
    \begin{equation}\label{eq:Bvariance}
    \sum_{m \in B} a_m^2 \ge \frac{1}{2}\sum_{m=1}^N a_m^2 = \frac{1}{2}\sum_{m = 1}^{n-1} m^2 L_m \ge 2n^2\log^{3+\eps}(n) L_n.
    \end{equation}
    Let $F_n$ be the event that $X_N$ lies in one of the congruence classes $-C, \dots, C$ modulo $n$. By Lemma~\ref{thm: modular ELO}  and the independence of $S_A$ and $S_B$, for any integer $x$ we have
    \[ \mathbb{P}(S_A \equiv x-S_B \pmod{n} \mid S_B) \le \frac{2}{n} + \sqrt{\frac{2}{\pi |A|}} < \frac{3}{n},\]
so
\[\mathbb{P}(F_n \mid S_B) \le \frac{3(2C+1)}{n}.\]
    Note that 
    \[\{|X_N| \leq 2\sqrt{M}\} \subseteq
    \{|S_A| \geq 2\sqrt{M}\} \cup \{ |S_{B}| \leq 4\sqrt{M}\},\] 
    and $E_n$ can only occur if  $F_n$ occurs. Hence, by a union bound, the probability of the event $E_n \cap \{|X_N| \leq 2\sqrt{M}\}$ is at most
    \[
    \mathbb{P}\big(|S_A| \ge 2\sqrt{M}\big) + \mathbb{P}\big(|S_B| \le 4\sqrt{M})\,\mathbb{P}(F_n \mid |S_B| \le 4\sqrt{M} \big).
    \]
    We will show that
    $\mathbb{P}(|S_A| \geq 2\sqrt{M}) = O(1/n^{2})$ and that $\mathbb{P}( |S_{B}| \leq 4 \sqrt{M}) = O(1/\log^{1+\eps/2}(n))$. Together these imply that
    \[\mathbb{P}(E_n \cap \{|X_N| \le 2 \sqrt{M}\}) =  O\left(\frac{1}{n^2} + \frac{1}{n \log^{1+\eps/2} (n)}\right) = O\left(\frac{1}{n \log^{1+\eps/2} (n)}\right),\]
    so by~\eqref{eq:largebound} we have
    \[\mathbb{P}(E_n) = \mathbb{P}(E_n \cap \{|X_N| \le 2\sqrt{M}\}) + \mathbb{P}(E_n \cap \{|X_N| > 2\sqrt{M}\})  = O\left(\frac{1}{n \log^{1+\eps/2}(n)}\right).\]
    Hence, $\sum_{n=1}^\infty \mathbb{P}(E_n) < \infty$ as required.

    Let us bound the probability that $|S_{B}| \leq 4\sqrt{M}$. Using the definition of $M$ and ~\eqref{eqn:assump2} we have
    \[ \frac{4\sqrt{M}}{\sqrt{\sum_{i\in B} a_i^2}} \ge \sqrt{\frac{16M}{\sum_{i=1}^N a_i^2}} = \sqrt{\frac{16\sum_{m=1}^{n-1} m^2 L_m}{\log^{2+\eps} (n) \sum_{m=1}^{n-1} m^2 L_m}} = \frac{4}{\log^{1+\eps/2} (n)}.\]
    Using the Berry--Esseen Theorem (see Theorem~\ref{thm:BerryEsseen}), we have that 
    \begin{align*}
        \mathbb{P}(|S_{B}| \leq 4\sqrt{M}) &=  \mathbb{P}\left(\frac{|S_{B}|}{\sqrt{\sum_{i \in B} a_i^2}} \leq \frac{4\sqrt{M}}{\sqrt{\sum_{i \in B} a_i^2}}   \right)\\
        &\leq \mathbb{P}\left(\frac{|S_{B}|}{\sqrt{\sum_{i \in B} a_i^2}} \leq \frac{4}{  \log^{1+\eps/2}(n) }   \right)\\
        &\leq \Phi\left(\frac{4}{  \log^{1+\eps/2}(n) } \right) - \Phi \left( - \frac{4}{ \log^{1+\eps/2}(n) } \right) + O\left( \frac{\sum_{i \in B} a_i^3}{\left(\sum_{i \in B} a_i^2\right)^{3/2}}\right).
    \end{align*}
    By equation~\eqref{eqn:assump2} the error term here is suitably small:
    \[
    \frac{\sum_{i \in B} a_i^3}{\left(\sum_{i \in B} a_i^2\right)^{3/2}}
    \le \frac{(n-1)\sum_{i \in B} a_i^2}{\left(\sum_{i \in B} a_i^2\right)^{3/2}}
    \leq \frac{n}{\left(\frac{1}{2}\sum_{m=1}^{n-1} m^2 L_m \right)^{1/2}} \le \frac{\sqrt{2}}{n \log^{(3+\eps)/2} (n)}.
    \]
    Hence,
    \[   \mathbb{P}(|S_B| \le 4\sqrt{M}) = O\left(\frac{1}{\log^{1+\eps/2} (n)}\right).
    \]
    Finally, we need to bound $\mathbb{P}(|S_A| \geq 2\sqrt{M})$. We have \[M \ge n^4 \log n,\] so $2\sqrt{M} \ge 2n^2\sqrt{\log n}$.  From the definition of $A$, we have $a_i \le n$ for all $i \in A$, and $|A| = n^2$ so  $\sum_{i \in A} a_i^2 \leq n^4$. Therefore, by Lemma~\ref{lem:rademacher-concentration}, we have
    \phantom{\qedhere}
    \[
    \mathbb{P}(|S_A| \ge 2\sqrt{M}) \le 2\mathbb{P}(S_A \geq 2 n^2\sqrt{\log(n)}) \leq 2e^{-2 \log(n)} = \frac{2}{n^2}.\tag*{\qed}
    \]  
    \end{proof}

    Given the theorem above, Corollary~\ref{cor:logsquared} follows easily. Indeed, we only need to show that $a_n = \lfloor\log^{1+\eps}(n)\rfloor$ satisfies the appropriate conditions.

    \logsq*
   
    \begin{proof}
        Clearly, $(a_n)$ is a non-decreasing sequence of integers, so we can apply Theorem~\ref{thm:ints} provided the $L_n$ satisfy the necessary conditions.
       The $i$th step size $a_i$ equals $n$ exactly when $e^{n^{1/\alpha}} \leq i < e^{(n+1)^{1/\alpha}}$, so either
       $L_n = \big\lfloor e^{(n+1)^{1/\alpha}} - e^{n^{1/\alpha}} \big\rfloor$ or $L_n =  \big\lceil e^{(n+1)^{1/\alpha}} - e^{n^{1/\alpha}} \big\rceil$.
       By the Mean Value Theorem, there is $x \in (n, n+1)$ such that 
       \[e^{(n+1)^{1/\alpha}} - e^{n^{1/\alpha}} = \frac{e^{x^{1/\alpha}}}{\alpha x^{1-1/\alpha}}.\]
       This is an increasing function of $x$ for sufficiently large $x$, so for sufficiently large $n$, (say $n \geq n_0$) we have
       \[ \frac{e^{n^{1/\alpha}}}{\alpha n^{1 - 1/\alpha}} -1 \le L_n \le \frac{e^{(n+1)^{1/\alpha}}}{\alpha (n+1)^{1 - 1/\alpha}} +1.\] This implies that $L_{n-1} \geq 2n^2$ for large enough $n$, so equation~\eqref{eqn:assump1} is satisfied.
       Also,
       \begin{align*}
           \sum_{i=1}^{n-1} i^2 L_i &\geq \sum_{i=n_0}^{n-1} i^2 \left(\frac{e^{i^{1/\alpha}}}{\alpha i^{1-1/\alpha}} - 1\right) \\
           &\geq \frac{1}{\alpha} \int_{n_0}^{n-1} x^{1 + 1/\alpha} e^{x^{1/\alpha}} dx - n^3\\
           &= \Omega\big((n-1)^{2}e^{(n-1)^{1/\alpha}}\big).
       \end{align*}
       Comparing this with $n^4 \log^{3+\eps} n $ and with  $4n^2 \log^{3+\eps}(n)  \big(\frac{e^{(n+1)^{1/\alpha}}}{\alpha (n+1)^{1-1/\alpha}} + 1\big)$, and recalling that $\alpha > 1$,
       we see that equation~\eqref{eqn:assump2} is also satisfied for sufficiently large $n$. Therefore, we can apply Theorem~\ref{thm:ints} to finish the proof.      
    \end{proof}

\section{Unbounded step sequences whose gaps tend to zero}
\label{sec: topological}

In this section we prove Theorem~\ref{thm: topological recurrence}, which asserts that if the step sizes of a Rademacher random walk are unbounded with gaps converging to zero, the walk is either transient or topologically recurrent.

\begin{lemma}\label{lem:coupling_game}
  Suppose $(a_n)_{n \geq 1}$ is a sequence such that $\limsup_{n \to \infty} a_n = \infty$ and $|a_n - a_{n-1}| \to 0$ as $n \to \infty$. 
  Let $X$ and $X'$ be Rademacher random walks with step sizes given by $(a_n)$ started at time $k$ at locations $X_k = N$ and $X'_k = N + d$ for some $d \neq 0$. Then, for any $\eps > 0$, the walks $X'$ and $X'$ can be coupled so that a.s.~$\lim_{n \to \infty} (X_n - X'_n)$ exists and lies in $[0, \eps]$.
\end{lemma}
\begin{proof} 
Think of the problem of coupling $X$ and $X'$ as a game as follows. Just before time $i$, we know $\epsilon_1, \dots, \epsilon_{i-1}$ and $\epsilon'_1, \dots, \epsilon'_{i-1}$ and we must choose either to couple $\epsilon'_i$ and $\epsilon_i$ so that $\epsilon'_i = \epsilon_i$, or to couple them so that $\epsilon'_i = -\epsilon_i$. Once we have made our choice, the value of $\epsilon_i$ is revealed, and $\epsilon_i'$ is either $\epsilon_i$ or $-\epsilon_i$ depending on the choice that we made. We win the game if for some $i \ge k$ we achieve $X_i - X_i' \in [0,\eps]$. The proof consists of a strategy for winning this game eventually with probability $1$. Once we have won the game we may couple all subsequent signs to be equal, so that $X_i - X_i'$ stabilizes.

Our strategy is organised as a sequence of episodes. In each episode we will win with probability at least $1/4$, conditional on all the outcomes in previous episodes. For $i \ge 1$, episode $i$ will begin at time $m_{i-1}+1$ and end at time $m_i$, where $m_0 = k$. To describe episode $i$ for any $i \ge 1$, assume we know $m_{i-1}$ $X_{m_{i-1}}$ and $X'_{m_{i-1}}$, and assume we have not yet won the game at time $m_{i-1}$. Let $d_i = X_{m_{i-1}} - X'_{m_{i-1}}$. In particular, $d_1 = d$. Let $\delta_i = \min(\epsilon, |d_i|)$. Note that $\delta_i > 0$ since if $d_i = 0$ then we have already won the game. Choose $n_i \geq m_{i-1}$ sufficiently large that for all $n \ge n_i$ we have $|a_n - a_{n-1}| < \delta_i/2$. Let
\[x_i = \begin{cases} d_i/2 & \text{ if $d_i > 0$},\\ -d_i/2 + \delta_i/2, & \text{ if $d_i \le 0$.} \end{cases} \]
Since $\limsup_{n \to \infty} a_n = \infty$, we may find $m_i > n_i$ such that 
\[a_{m_{i}} - a_{n_i} \in [x-\delta_i/2, x].\] Now choose to couple the signs $\epsilon_i$ and $\epsilon_i'$
driving the movements of the walks $X$ and $X'$ to be equal for each $i$ in the range $m_{i-1} \le i \le n_i -1$ and for $n_i+1 \le i \le m_i - 1$. Choose to couple $\epsilon_{n_i} = - \epsilon'_{n_i}$ and $\epsilon_{m_i} = - \epsilon'_{m_i}$.   

There are four possible options for $X'_{m_i} - X_{m_i}$, each having probability $1/4$ conditional on the outcomes preceding episode $i$:
\[
X'_{m_i} - X_{m_i} = \begin{cases}
    d_i + 2 (a_{m_i} + a_{n_i}) & \epsilon_{n_i} = -1, \epsilon_{m_i} = -1,\\
    d_i + 2 (a_{m_i} - a_{n_i}) & \epsilon_{n_i} = +1, \epsilon_{m_i} = -1,\\
    d_i - 2 (a_{m_i} - a_{n_i}) & \epsilon_{n_i} = -1, \epsilon_{m_i} = +1,\\
    d_i - 2 (a_{m_i} + a_{n_i}) & \epsilon_{n_i} = +1, \epsilon_{m_i} = +1.
\end{cases}
\]
If $d_i > 0$ then $d_i - 2 (a_{m_i} - a_{n_i}) \in [0, \eps]$. If $d_i < 0$ then $d_i + 2 (a_{m_i} - a_{n_i}) \in [0, \eps]$. Hence, with probability at least $1/4$, we have $X'_{m_i} - X_{m_i} \in [0,\eps]$, in which case we win the game no later than time $m_i$.

Carrying out the procedure above repeatedly until success, we will almost surely win after finitely many episodes, since in each episode we win with probability at least $1/4$, conditional on the outcomes in previous episodes.
\end{proof}

Using this coupling, we show that if the walk hits the interval $[a,b]$ i.o.\ with positive probability, then the probability that it hits a second interval i.o.\ is also positive.

\begin{corollary}\label{cor: m}
Suppose $(a_n)_{n \geq 1}$ is a sequence such that $\limsup_{n \to \infty} a_n = \infty$ and $|a_n - a_{n-1}| \to 0$ as $n \to \infty$. Let $X = (X_n)_{n \geq 1}$ be the Rademacher random walk with step sizes $(a_n)$.
Let $a< b$ and $e < f$ and take $m \in \mathbb{N}$ such that $m > (b-a)/(f-e)$.  If $\mathbb{P}(X_n \in [a,b] \text{ i.o.}) \ge p$, then $\mathbb{P}(X_n \in [e,f] \text{ i.o.}) \ge p/m$.
\end{corollary}
\begin{proof}
Divide the interval $[a,b]$ into $m$ equal intervals. By a union bound, at least one such interval $[a',b']$ is visited infinitely often with probability at least $p/m$. We have $[a'+t,b'+t] \subset [e,f)$, where $t = e-a'$. If $t = 0$ there is nothing to do, so let us assume that $t \neq 0$ and further that $t > 0$. The case $t < 0$ is similar. Apply the coupling of Lemma~\ref{lem:coupling_game} starting at time $0$ at $X_0 = 0$ and $X'_0 = t$, taking $\eps = f-(b'+t)$. Then the standard Rademacher random walk $X''$ defined by $X''_n = X'_n - t$ visits $[e,f]$ infinitely often if $X$ visits $[a',b']$ infinitely often, and this occurs with probability at least $p/m$.
\end{proof}
The proof of Theorem~\ref{thm: topological recurrence} follows.
\begin{proof}\label{lem: transient or dense}
Suppose $X$ is not transient. Then there exists $C < \infty$ such that 
\[\mathbb{P}(|X_n| < C \text{ i.o.}) > 0.\]
Apply Corollary~\ref{cor: m} taking $[a,b] = [-C,C]$ and $p = \mathbb{P}(|X_n| < C \text{ i.o.})$, to see that whenever $e < f$ we have 
\[\mathbb{P}(X_n \in [e,f] \textup{ i.o.})  \ge \frac{p}{\lceil 2C/(f-e) \rceil} > 0.\]
Now suppose (for a contradiction) that for some interval $[g,h]$, we have 
\[\mathbb{P}(X_n \in [g,h] \textup{ i.o.}) <1.\]
Then (by a standard martingale argument) there exists a finite $k'$ and two sequences of signs $\beta_1, \dots, \beta_{k'}$ and $\gamma_1, \dots, \gamma_{k'}$ such that
\[\mathbb{P}(X_n \in [g,h] \text{ i.o.} \mid \epsilon_1 = \beta_1, \dots, \epsilon_k = \beta_{k'}) > 2/3 \] and
\[\mathbb{P}(X_n \in [g,h] \text{ i.o.} \mid \epsilon_1 = \gamma_1, \dots, \epsilon_k = \gamma_{k'}) < 1/3.\]
Take $m=2$ in Corollary~\ref{cor: m}, applied to the walk with step sizes $a_{{k'}+1}, a_{{k'}+2}, \dots$, and with \[[a,b] = [g - (\beta_1 a_1 + \dots + \beta_{k'} a_{k'}), h - (\beta_1 a_1 + \dots + \beta_{k'} a_{k'})]\] and 
\[ [e,f] = [g - (\gamma_1 a_1 + \dots + \gamma_{k'} a_{k'}), h - (\gamma_1 a_1 + \dots + \gamma_{k'} a_{k'})]\]
to obtain a contradiction.
\end{proof}

\section{Recurrent Rademacher walks where the step sizes grow arbitrarily quickly}
\label{sec:fast}

In this section we prove Theorem~\ref{thm:fast} and Theorem~\ref{thm:fastincreasing}, both of which show the existence of sequences of step sizes which grow arbitrarily quickly yet give recurrent Rademacher random walks.
The constructions used in the proofs of Theorem~\ref{thm:fast} and Theorem~\ref{thm:fastincreasing} both work by considering a suitable two-dimensional random walk and understanding the range of the second coordinate at the times when the first coordinate is zero, but the proof of Theorem~\ref{thm:fast} is much simpler.

\fast*

\begin{proof}
    We define the sequence $(a_n)$ in blocks, starting from the empty sequence.
    Suppose that $a_1, \dots, a_N$ have already been chosen, and let $M = \sum_{i=1}^N a_i$. 
    Since the two-dimensional simple symmetric random walk is recurrent, we can choose some $L$ such that the probability that the two-dimensional simple symmetric random walk (SSRW) has hit every point in $\{(0, y) : y \in [-M, M]\}$ by time $L$ is at least $1/2$. 
    Now we choose $r$ such that $r \geq f(2L + N)$ and define the next $2L$ steps to alternate between $r + 1$ and $r$.
    This sequence clearly satisfies the necessary growth condition and it remains to prove that the sequence is weakly recurrent.

    Let $E_k$ be the event that the walk hits 0 in the $k$th block.
    We claim that uniformly for any outcome on the preceding blocks, the probability of $E_k$ is at least $1/2$. Suppose that the $k$th block starts at $a_{N+1}$ and that the walk is at $m$ immediately before the $k$th block, i.e.~$X_N = m$. Pair up consecutive steps in the $k$th block and consider the walk $Y = (Y_n)$ where $Y_n = X_{N + 2n}$. 
    This starts at $m$ and takes steps of $\pm (2r+1), \pm 1$ each with probability $1/4$.
    As before, we define a two dimensional random walk $(x_n, y_n)$ by setting $(x_0, y_0) = (0,0)$, and then inductively defining $(x_n, y_n)$ for $n = 1, \dots, L$ by 
    \[
    (x_n, y_n) = (x_{n-1},y_{n-1}) + \begin{cases} 
    (1,0)  & \text{ if $Y_{n} - Y_{n-1} = 2r+1$,}\\
    (-1,0)  & \text{ if $Y_{n} - Y_{n-1} = - (2r+1)$,} \\
    (0,1)  & \text{ if $Y_{n} - Y_{n-1} = 1$,} \\
    (0,-1)  & \text{ if $Y_{n} - Y_{n-1} = -1$.}\end{cases} 
    \]
    Clearly, if the walk $(x_n, y_n)$ hits the point $(0, -m)$, then the walk $X$ has hit zero and by our choice of $L$ this happens with probability at least $1/2$.
    To finish the proof that the walk is recurrent, we can apply the conditional Borel--Cantelli lemma to see that almost surely infinitely many of the events $E_k$ occur.
\end{proof}

We now turn to the proof of Theorem~\ref{thm:fastincreasing}. We will again consider the times when the first coordinate of a two-dimensional random walk is 0, but we will have to work with a more complicated two-dimensional random walk and we will need the following lemma.

\begin{lemma}\label{lem: 2D RW}
Define for each $n \ge 1$
\[ c_{n} = \sum_{m=1}^{n-1} \frac{m^{-3/2}}{\sqrt{1+\log m}}.\]
Consider the two-dimensional Rademacher random walk $(Y_n,Z_n)_{n\ge 0}$ with $n^{th}$ step $\pm(1, c_n)$, starting at $(Y_0,Z_0) = (0,0)$. Almost surely the set $\{Z_n: Y_n = 0\}$ is dense in $\mathbb{R}$. 
\end{lemma}

From this lemma it is relatively straightforward to prove Theorem~\ref{thm:fastincreasing}.

\fastincreasing*
\begin{proof}
Define for each $n \ge 1$
\[ c_{n} = \sum_{m=1}^{n-1} \frac{m^{-3/2}}{\sqrt{1+\log m}}.\]
Note that $\sum_{m=1}^\infty \frac{m^{-3/2}}{\sqrt{1+\log m}} < \infty$, so $c_n \nearrow c_\infty$ as $n \to \infty$ where $c_\infty < \infty$. Also,
\[c_{n+1} - c_n = \frac{n^{-3/2}}{\sqrt{1+\log n}}\,.\]

We will define the step size sequence $(a_n)_{n \ge 1}$ as the concatenation of blocks of the form $(x_j + c_0, x_j + c_1, \dots, x_j + c_{n_j})$, for $j \ge 1$, where for each $j$ we choose $n_j$ and then $x_j$ suitably large given the previous choices. Let $M_j = \sum_{n=1}^{n_1 + \dots + n_{j-1}} a_n$ be the sum of all the terms in the blocks preceding the $j^{\textup{th}}$ block. According to Lemma~\ref{lem: 2D RW}, we may choose $n_j$ so that with probability at least $1/2$, the walk $(Y_n,Z_n)_{n=0}^{n_j}$ visits the $(1/j)$-neighbourhood of each point in $\{0\} \times [-2M_j,2M_j]$. Then choose $x_j$ so large that $x_j \ge f(n_1 + \dots + n_j)$.  This ensures $a_n \ge f(n)$ for all $n$. Note that $n_j \to \infty$ as $j \to \infty$ and hence $M_j \to \infty$ also. For any $t \in \mathbb{R}$, once $M_j +1/j > t$, we have that conditional on $X_{n_1 + \dots + n_{j-1}}$, the Rademacher walk $X_n$ visits the interval $(t-1/j, t+1/j)$ during block $j$ with probability at least $1/2$. Hence, $X$ is topologically recurrent by the conditional Borel--Cantelli lemma.
\end{proof}

It remains to prove Lemma~\ref{lem: 2D RW}.

\begin{proof}[Proof of Lemma~\ref{lem: 2D RW}.]
Consider the sequence of successive times $ 0 =\tau_1 < \tau_2 < \tau_3 < \dots$ at
which $Y_n$ = 0. This is almost surely an infinite increasing sequence tending to $\infty$ since the walk $Y = (Y_n)_{n \geq 0}$ on its own is a simple symmetric random walk on $\mathbb{Z}$, which is recurrent. The distribution of each increment $Z_{\tau_{i+1}} - Z_{\tau_i}$ conditional on the earlier increments has a symmetric distribution. 
Therefore we can condition on the increment sizes $|Z_{\tau_{i+1}} - Z_{\tau_i}|$ and obtain a \emph{random} Rademacher walk $\left(Z_{\tau_i}\right)_{i \ge 1}$ with random step sizes $b_i := |Z_{\tau_{i+1}} - Z_{\tau_i}|$ for $i \ge 1$. We claim that almost surely $b_i \to 0$ as $i \to \infty$ and $\sum_{i=1}^\infty b_i^2 = \infty$. Given this claim, we may apply Lemma~\ref{L3} to see that almost surely the Rademacher walk $(Z_{\tau_i})_{i \ge 1}$ is topologically recurrent, which is to say that almost surely $\{Z_n \,:\, Y_n = 0\}$ is dense in $\mathbb{R}$.

The key observation is that since the sign of $Y_n$ is constant for $n$ in the interval $[1+\tau_{i},\tau_{i+1}-1]$, and $Y_n = 0$ when $n \in \{\tau_i,\tau_{i+1}\}$, and $c_{n+1} - c_n > 0$ for all $n$, we have
\begin{align*}
|Z_{\tau_{i+1}} - Z_{\tau_i}| &= \left|\sum_{n=1+\tau_{i}}^{\tau_{i+1}} (Y_n-Y_{n-1}) c_n\right| = 
\left|-\sum_{n = 1+\tau_i}^{\tau_{i+1}} Y_n(c_{n+1} - c_n) \right|\\
&=\sum_{n = 1+\tau_i}^{\tau_{i+1}-1} |Y_n|(c_{n+1} - c_n) = \sum_{n = 1+\tau_i}^{\tau_{i+1}-1} |Y_n|\frac{n^{-3/2}}{\sqrt{1 + \log n}}.
\end{align*}
Since $Y_n =0$ when $n \in \{\tau_i,\tau_{i+1}\}$, the final sum above is unchanged if we replace the lower limit by $n=\tau_i$ or the upper limit by $\tau_{i+1}$.

Let us use the Koml\'os--Major--Tusn\'ady coupling to couple the simple symmetric random walk $Y_n$ to a standard Brownian motion $(B_s)_{s \ge 0}$. This is a coupling with the property that for every $\alpha > 0$ there is a positive constant $c_\alpha$ such that for all $n$ one has
\[
\mathbb{P}\left(\max_{1 \le j \le n} \frac{|B_j - Y_j|}{\log n} > c_\alpha\right) < c_\alpha n^{-\alpha}.
\]
(See \cite[Theorem 7.1.1]{lawler2010random}.)
Taking $\alpha > 1$ and using Borel--Cantelli it follows that almost surely
\[\limsup_{s \to \infty} \frac{|B_s - Y_{\lfloor s \rfloor}|}{\log s} < \infty.\]
Since $|Y_{n+1} - Y_n| = 1$ for all $n$, we also have
\[\limsup_{s \to \infty} \frac{|B_s - Y_{\lceil s \rceil}|}{\log s} < \infty.\]
Let $E$ be the (almost surely finite) random variable defined by
\[ E = \sup_{s \ge 1}\frac{\max(|B_s - Y_{\lfloor s\rfloor}|, | B_s - Y_{\lceil s\rceil}|)}{1+\log s}.\]
Then, for each $i \ge 2$, we have $\tau_i \ge 2$ so the function $x^{-3/2}/\sqrt{1+\log x}$ is decreasing on the interval $[\tau_i,\tau_{i+1}]$ and hence 
\[
\sum_{n = \tau_i}^{\tau_{i+1}-1} |Y_n| \frac{n^{-3/2}}{\sqrt{1+ \log n}}
  >  \int_{\tau_i}^{\tau_{i+1}} \left(\left|B_s\right| - (1+\log s)E\right) \frac{s^{-3/2}}{\sqrt{1+\log s}} \,\textup{d}s 
\]
and
\[
\sum_{n = 1+\tau_i}^{\tau_{i+1}} |Y_n| \frac{n^{-3/2}}{\sqrt{1+ \log n}}
  <  \int_{\tau_i}^{\tau_{i+1}} \left(\left|B_s\right| + (1+\log s)E\right) \frac{s^{-3/2}}{\sqrt{1+\log s}} \,\textup{d}s 
.
\]

Note that $\int_1^\infty \left(\sqrt{1+\log s}\right) s^{-3/2} \,\textup{d}s < \infty$, so the sequence of errors 
\[\left(\int_{\tau_i}^{\tau_{i+1}} \left(\sqrt{1+\log s}\right)s^{-3/2}E\,\textup{d}s\right)_{i \ge 2}\]
is almost surely in $\ell^1(\mathbb{N})$ and hence also in $\ell^2(\mathbb{N})$.

Therefore, to show that the sequence $(b_i)_{i \ge 1} = \left(\left|Z_{\tau_{i+1}} - Z_{\tau_i}\right|\right)_{i \ge 1}$ almost surely tends to $0$ but does not lie in $\ell^2(\mathbb{Z})$, it suffices to do the same for the sequence $(I_i)_{i \ge 1}$ defined by 
\[ I_i:=\int_{\tau_{i}}^{\tau_{i+1}} \left| B_s \right| \frac{s^{-3/2}}{\sqrt{1+\log s}} \,\textup{d}s.\]
We make the change of variable $s = e^t$, to get
\[
I_i
 = 
\int_{\log(\tau_i)}^{\log(\tau_{i+1})} \left| B_{e^t}\right| \frac{e^{-t/2}}{\sqrt{t+1}} \,\textup{d}t 
 = 
\int_{\log(\tau_i)}^{\log(\tau_{i+1})} \frac{|W_t|}{\sqrt{t+1}}\,\textup{d}t,
\]
where the process $$W_t : = B_{(e^t)} e^{-t/2}$$ is a stationary Ornstein--Uhlenbeck process whose stationary distribution $\pi$ is Gaussian with mean $0$ and variance $1$. For this identity in law, see \cite[\S8.5.1]{stroock2011probability}, in which $W$ is called an \emph{ancient} Ornstein--Uhlenbeck process. $W$ satisfies the SDE
\[ dW_t = -\frac{1}{2}W_t\,dt + dB'_t,\] where $(B'_t)_{t \in \mathbb{R}}$ is another (two-sided) standard Brownian motion.

The rough idea now is that large increments $I_i$, exceeding a positive constant size, correspond to increasingly large excursions of $W$ from $0$, which only occur finitely often, almost surely, but on the other hand the large excursions of $W$ from $0$ whose integral is at least a positive constant occur sufficiently regularly to give a subsequence of $(I_i)$ whose sum diverges.  Some care is needed to make this precise. 

To show that $I_i \to 0$ almost surely as $i \to \infty$, we consider the excursions of $W_t $ above $-1$ and the excursions of $W_t$ below $1$. The times $\log(\tau_i)$ are times at which $|Y_{e^t}| = 0$ and hence $|B_{e^t}| \le (1+t)E$. 
There exists a random time $t_0 < \infty$ such that 
\[(1+t)e^{-t/2}\,E < 1/2 \text{ for $t \ge t_0$.}\] 
For $i \ge i_0 :=\lceil e^{t_0} \rceil$, we have $\tau_i \ge i$ so for all $t > \log(\tau_i)$ we have $(1+t)e^{-t/2} E < 1/2$. For each $i \ge i_0$, on the interval $\tau_i < n < \tau_{i+1}$, either all $Y_n > 0$, in which case 
\[ W_t > -(1+t)e^{-t/2} E > -1/2 \text{  for all } t \in [\log(\tau_i),\log(\tau_{i+1})],\] or all $Y_n < 0$, in which case 
\[W_t < (1+t)e^{-t/2} E < 1/2 \text{ for all } t \in [\log(\tau_i),\log(\tau_{i+1})].\]

It follows that for each $i \ge i_0$, the integral $I_i$ is dominated either by an integral of the form 
\[ \frac{1}{\sqrt{a_i+1}}\int_{a_i}^{b_i} (W_t + 1) \textup{d}t,\]
where $[a_i,b_i]$ is an excursion interval of $W$ above the level $-1$ which reaches the level $-1/2$, or by an integral of the form
\[\frac{1}{\sqrt{a_i+1}}\int_{a_i}^{b_i} (-W_t - 1) \textup{d}t\]
where $[a_i,b_i]$ is an excursion interval of $W$ below the level $1$ which reaches the level $1/2$. Since $\log(\tau_i) \to \infty$ as $i \to \infty$, and $W$ hits each of $-1$ and $1$ at an unbounded set of times almost surely, we have that $a_i \to \infty$ as $i \to \infty$.

The law of the iterated logarithm for $B_s$ as $s \to \infty$ corresponds to a simpler statement about the maximal growth of the stationary Ornstein--Uhlenbeck  process (see \cite[eq. (8.5.2)]{stroock2011probability}):
\[\limsup_{t \to \infty} \frac{|W_t|}{\sqrt{\log t}} = \sqrt{2} \text{ a.s.}\]
Hence, there is a random time $t_1$ such that 
\begin{equation} \label{eq:Wbound}
|W_t| < 2 \sqrt{\log t} \text{ for all } t \ge t_1.
\end{equation}

It is known that the hitting times in one-dimensional Ornstein--Uhlenbeck processes have exponential tails. 
The exact rate of exponential decay is given in~\cite[Corollary~3.1]{alili2005representations} in terms of Weber's parabolic cylinder functions, which arise from solving the Sturm-Liouville problem associated with the Ornstein--Uhlenbeck process on a half-infinite interval with the Dirichlet boundary condition.
From the exponential tails of hitting times, it follows that the successive excursions of $W$ above $-1$ that reach $-1/2$ have lengths that form an i.i.d.~sequence with exponential tails. Indeed, the length of each one is the sum of two independent hitting times: the hitting time of $-1/2$ starting at $-1$, and the hitting time of $-1$ starting at $-1/2$. 
The same applies to successive excursions below $1$ that reach $1/2$.
Among both such types of excursion, we now consider only those whose length is at least $1$, since by~\eqref{eq:Wbound} excursions shorter than this after time $t_1$ can only produce increments $I_i$ that tend to $0$.  For each type of excursion, the $n^{th}$ instance of the remaining long excursions starts at a time that is at least $n$, since they all have length at least $1$ and are disjoint (although excursions of the two different kinds can overlap). Hence, if $L_t$ is the length of the longest excursion of either kind up to time $t$, then 
\[\sup_{t \ge 2} L_t/\log t < \infty \text{ a.s.}\]
because of the exponential tail bound that we noted above. Combining this bound on the excursion lengths with the bound~\eqref{eq:Wbound} on their heights, we get a bound on the largest integral $\int_{a_i}^{b_i}|W_t+1|\,\textup{d}t$ where $b_i - a_i \ge 1$ and $a_i \le t$: it is no more than \[2\sqrt{\log t} \log t \,\sup_{t \ge 2} L_t/\log t,\]
which is $o(\sqrt{t+1})$ as $t \to \infty$.
It follows that $|Z_{\tau_{i+1}} - Z_{\tau_i}| \to 0$ as $i \to \infty$.

It remains to show that the sequence $\left(|Z_{\tau_{i+1}} - Z_{\tau_i}|\right)_{i \ge 1}$ almost surely does not lie in $\ell^2(\mathbb{N})$.

Let us call an excursion of $W$ above $-1$ a \emph{good excursion} if it contains an excursion above $1$ whose integral is at least $1$. 

Consider a good excursion of $W_t$ above $-1$ that starts at time $s$ and finishes at time $t$, where $t_0 \le s < t - 1$, with a subinterval $(s',t') \subset (s,t)$ such that $W_r \ge 1$ for all $r \in (s',t')$ and $\int_{s'}^{t'} W_r \,\textup{d}r > 1$. Then there is an excursion of $Y$ above $0$, say from time $\tau_i$ to time $\tau_{i+1}$, where \[s < \log \tau_i < s' < t' < \log \tau_{i+1} < t\]
and 
\[ I_i  = \int_{\log(\tau_i)}^{\log(\tau_{i+1})} \frac{W_t}{\sqrt{t+1}}\,\textup{d}t > \frac{1}{\sqrt{t'+1}}\int_{s'}^{t'} W_t\,\textup{d}t > \frac{1}{\sqrt{t'+1}}\,. \]
 The number of disjoint good excursions that lie entirely between times $4^{k-1}$ and $4^k$ grows as $\Theta(4^k)$, almost surely as $k \to \infty$. (Again, this follows from the fact that the hitting times of $1$ starting from $-1$ and vice versa in the Ornstein--Uhlenbeck process have exponential tails, together with the Markov property of the Ornstein--Uhlenbeck process.) Hence, the sequence $\left(I_i\right)_{i \ge 1}$ almost surely does not lie in $\ell^2(\mathbb{N})$.
 This completes the proof of  Lemma~\ref{lem: 2D RW}.
\end{proof}

\section*{Acknowledgements}
 The research of E.C.~and T.J.~is
 supported by the Heilbronn Institute for Mathematical Research. We thank Stanislav Volkov and Zemer Kosloff for helpful conversations about aspects of this work.

\bibliographystyle{abbrv}
\bibliography{refs}
\end{document}